\documentclass[preprint,12pt]{elsarticle}
\usepackage[utf8]{inputenc}
\usepackage{amssymb}
\usepackage{amsmath}
\usepackage{amsthm}
\usepackage{xcolor}
\usepackage{enumitem}
\usepackage{tikz}
\usepackage{hyperref}
\usepackage{ulem}

\def\F{\mathbb{F}_q}

\newtheorem{theorem}{Theorem}[section]
\newtheorem{lemma}[theorem]{Lemma}
\newtheorem{corollary}[theorem]{Corollary}
\newtheorem{proposition}[theorem]{Proposition}
\newtheorem{remark}[theorem]{Remark}

\journal{Finite Fields and Their Applications}

\begin{document}
\begin{frontmatter}

\title{Minimal value set binomials and Frobenius nonclassical curves}
\author[1]{Tiago Aprigio}
\ead{tiagoaprigio@usp.br}
\author[2]{João Paulo Guardieiro}
\ead{joao.guardieiro@ufma.br}

\affiliation[1]{
    organization={Instituto de Ciências Matemáticas e Computação},
    addressline={Av. Trabalhador São-Carlense 400}, 
    postcode={13566-590}, 
    city={São Carlos},
    state={São Paulo},
    country={Brazil}
}

\affiliation[2]{
    organization={Centro de Ciências Exatas e Tecnologia da Universidade Federal do Maranhão},
    addressline={Av. dos Portugueses, 1966}, 
    postcode={65080-805}, 
    city={São Luís},
    state={Maranhão},
    country={Brazil}
}

\begin{abstract}
    In this paper, we characterize all minimal value set binomials over $\mathbb{F}_q$, that is, binomials whose size of the set of images is the smallest possible. With this information, we also classify all quadrinomial curves with separated variables that are $\mathbb{F}_q$-Frobenius nonclassical for the morphism of lines.
\end{abstract}


\end{frontmatter}

\section{Introduction}

Let $p$ be a prime number, $q$ be a power of $p$ and $\mathbb{F}_q$ be the finite field with $q$ elements. A simple argument involving the maximum number of roots of a polynomial $F \in \mathbb{F}_q[x]$ allows one to conclude that
$$\left\lfloor \frac{q - 1}{\deg F}\right\rfloor + 1 \leq \# V_F \leq q,$$
where $V_F = \{F(\alpha) : \alpha \in \mathbb{F}_q\}$ is the \textbf{value set} of the polynomial $F$. In this work, we are interested in polynomials that attain the lower bound for the cardinality of their value set. These are the so-called \textbf{minimal value set polynomials} (MVSPs).

This concept was introduced in \cite{carlitz61}: an extension of Waring's problem for prime fields leads to the question of determining a minimal bound for the value set of a polynomial. This question naturally leads to the problem of determining the polynomials attaining this minimal bound, i.e. MVSPs. Several authors have studied this concept by fixing the size of the finite field (\cite{borges-reis1, carlitz61, Mills1964}), constructing MVSPs from well-known polynomials (\cite{chow-et-al}) and by studying the possible value sets (\cite{borges-reis}), for example. To our knowledge, the present paper is the first work that classifies MVSPs by fixing the number of monomials, and we hope that this provides a new way to address the problem of the classification of MVSPs.

Another important work on this topic is \cite{borges-separated}. In this paper, the author provides a connection between MVSPs over $\mathbb{F}_q$ and $\mathbb{F}_q$-Frobenius nonclassical curves, that is, curves for which the image under the $\mathbb{F}_q$-Frobenius morphism of any smooth point lies on its tangent line. More precisely, the author showed that the components of a curve given by the plane equation $f(x) = g(y)$ are $\mathbb{F}_q$-Frobenius nonclassical if, and only if, $f$ and $g$ are MVSPs with the same value set.

Frobenius nonclassical curves were introduced by Stöhr and Voloch in \cite{StohrVol} and have been studied since then, as can be seen in \cite{borges, borges-homma, hefez}. In their paper, the authors presented a bound for the number of rational points of curves based on linear series (see \cite[Corollary 2.14]{StohrVol}). Roughly speaking, their bound for Frobenius classical curves is sharper than Hasse-Weil's bound. On the other hand, Frobenius nonclassical curves are a potential source of curves with many points (for example, the famous Hermitian curve is a maximal curve, and also a Frobenius nonclassical one). This illustrates the importance of classifying Frobenius (non)classical curves.

In his PhD thesis, \cite{guardieiro}, the second author of this work classified all Frobenius nonclassical trinomial curves, that is, curves given by a plane equation with three monomials. A natural extension of this problem is to study the quadrinomial curves, which is one of the problems addressed in this paper.

In order to classify some of these curves, we will first characterize all minimal value set binomials. With this information, and using the connection given in \cite{borges-separated}, we will then classify the Frobenius nonclassical quadrinomial curves that can be given, after a change of variables, by an equation with separated variables $f(x) = g(y)$. The main results of this paper are the following.
\\

\noindent \textbf{Theorem A.} 
A binomial over $\mathbb{F}_q$ is a minimal value set polynomial if, and only if, up to multiplication by a nonzero element of $\mathbb{F}_q$ it is given by one of the following expressions:
    \begin{enumerate}
        \item[$\mathrm{(i)}$] $x^a + \beta$, where $\beta \in \mathbb{F}_q^*$ and $a \mid (q - 1)$;
        \item[$\mathrm{(ii)}$] $x^{b + \ell(q - 1)} - x^b$, for some $b, \ell \in \mathbb{N}$;
        
        \item[$\mathrm{(iii)}$] $ x^{\frac{2(q-1)}{3}} + \beta x^{\frac{q-1}{3}}$, where $q = 2^{2n}$ and $\beta^3 = 1$.
        \item[$\mathrm{(iv)}$] $x^{q-1} + \beta x^{\frac{q-1}{2}}$, where $p> 2$ and $\beta^2=1$;
        \item[$\mathrm{(v)}$] $x^{p^{\ell}t} + \beta x^t$, $t = \frac{q - 1}{p^{s\ell} - 1}$, $\beta \in \mathbb{F}_{p^{s\ell}}^*$ and $N_{\mathbb{F}_{p^{s\ell}} / \mathbb{F}_{p^{\ell}}}(\beta) = (-1)^s$ with $s \geq 2$;
        \item[$\mathrm{(vi)}$] $x^{2t} + \beta x^t$, $p > 2$, $t = \frac{q - 1}{p^m - 1}$ and $\beta \in \mathbb{F}_{p^m}^*$. 
    \end{enumerate}
    
\noindent \textbf{Theorem B.} Let $\mathcal{C}$ be an irreducible quadrinomial curve over $\mathbb{F}_q$ with separated variables. Then $\mathcal{C}$ is $\mathbb{F}_q$-Frobenius nonclassical if, and only if, $\mathcal{C}$ is given by one of the following plane equations:
\begin{enumerate}
    \item[$\mathrm{(i)}$] $y^{d + q - 1} - y^d = \alpha x^{b + q - 1} - \alpha x^b,$ for suitable $b,d\in\mathbb{N}$;
    \item[$\mathrm{(ii)}$] $\mathbb{F}_4 \subseteq \mathbb{F}_q$ and $y^{q - 1} = x^{\frac{2(q - 1)}{3}} + x^{\frac{q - 1}{3}} + 1$;
    \item[$\mathrm{(iii)}$] $y^{(p^{\ell}+1)t} = x^{p^{\ell}t} + x^t + 1$, where $t = \frac{q - 1}{p^{2\ell} - 1}$;
    \item[$\mathrm{(iv)}$] $y^{q-1} = x^{4t} + x^{2t} + x^t$, where $q = 2^{3n}$ and $t = \frac{q - 1}{7}$;
    \item[$\mathrm{(v)}$] $y^{q-1} = x^{6t} +  x^{5t} + x^{3t}$, where $q = 2^{3n}$ and $t = \frac{q - 1}{7}$;
    \item[$\mathrm{(vi)}$] $y^{q-1} = x^{3t} + x^{2t} + x^t$, where $q = 2^{2n}$ and $t = \frac{q - 1}{3}$;
    \item[$\mathrm{(vii)}$] $y^{(p^{2\ell} + p^\ell + 1)t} = x^{p^{2\ell}t} + x^{p^\ell t} + x^t,$ where $\ell > 0$ and $t = \frac{q - 1}{p^{3\ell} - 1}$;
    \item[$\mathrm{(viii)}$] $y^{(p^{2 \ell} + p^\ell + 1)t} = x^{(p^{2 \ell} + p^\ell)t} + x^{(p^{2 \ell} + 1)t} + x^{(p^\ell + 1)t}$, where $t = \frac{q - 1}{p^{3\ell} - 1}$;
    \item[$\mathrm{(ix)}$] $y^{(p^\ell+1)t} = \alpha x^{(p^\ell+1)t} + x^{p^{\ell}t} + x^{t}$, where $t = \frac{q - 1}{p^{2\ell} - 1}$, and $\alpha \in \mathbb{F}_{p^{\ell}}$.
\end{enumerate}

We should mention that the binomials we classified in Theorem A are in accordance with \cite[Conjecture 1.1]{borges-reis}. The quadrinomial curves of items (iii) and (vii) - (ix) are also in accordance with \cite[Theorem 2.8]{borges-reis} (which was proved assuming their Conjecture). 

This paper is organized as follows: in Section \ref{section:background} we first provide the basic concepts we will need throughout the text. In Section \ref{section:MVSB} we will study the minimal value set binomials, and we prove Theorem A. Finally, in Sections \ref{section:quadrinomial1}, \ref{section:quadrinomial2} and \ref{section:quadrinomial3} we will study the quadrinomial curves with separated variables. The combination of the results of these final sections provides the proof of Theorem B.

\section{Background}\label{section:background}
In this section, we recall the main results we will need on this work: some results on minimal value set polynomials (MVSPs) and concepts on Frobenius nonclassical quadrinomial curves.

As mentioned before, a minimal value set polynomial (MVSP) $F \in \mathbb{F}_q[x]$ is the one attaining the lower bound for the cardinality of its value set, that is,
$$\# V_F = \left\lfloor \frac{q - 1}{\deg F}\right\rfloor + 1.$$

Let $p$ be the characteristic of $\mathbb{F}_q$. Since the polynomial $x^p$ induces a permutation in $\mathbb{F}_q$, one can see that an MVSP $F$ cannot be written as $G^p$, for any $G \in \mathbb{F}_q[x]$. In fact, this equality would provide $\deg F = p\cdot \deg G$ and $\# V_F = \# V_G$.

It was noted in \cite{carlitz61} that an MVSP $F(x)$ with $\# V_F \leq 2$ does not fit in a general pattern, and the authors classified such polynomials in the following way.

\begin{proposition}\label{prop:Vf-menor-que-2}
    If \( V_F = \{\alpha\} \), then \( F \) is an MVSP if, and only if, \( F \) is of the form
$$
F(x) = \alpha + (x^q - x) R(x),
$$
where \( R \in \mathbb{F}_q[x] \backslash\{0\} \). If $V_F = \{\alpha, \beta\}$, then $F$ is an MVSP if, and only if, 
$$
F(x) = \alpha \sum_{\mu \in \mathcal{S}} \left\{ 1 - (x - \mu)^{q - 1} \right\} + \beta \sum_{\mu \notin \mathcal{S}} \left\{ 1 - (x - \mu)^{q - 1} \right\},
$$
for some non-empty proper subset \( \mathcal{S} \) of \( \mathbb{F}_q \).
\end{proposition}

In order to analyze MVSPs whose value set has more than two elements, we will use the following result, which is a consequence of \cite[Equation (13) and Theorem 1]{Mills1964} and \cite[Theorem 2.2]{borges-separated}.

\begin{theorem}\label{MVSP-condition}
    Let $F(x) \in \mathbb{F}_q[x]$ be a nonconstant polynomial and $V_F = \{ \gamma_0, \gamma_1, \ldots, \gamma_r \} \subseteq \mathbb{F}_q$ be its value set. Assume that $r > 1$. For each $i \in \{0, \ldots, r\}$, set 
    $$L_i(x) = \prod_{\substack{\alpha \in \mathbb{F}_q\\F(\alpha) = \gamma_i}}(x - \alpha).$$
    Suppose that the $\gamma_i$'s are arranged in such a way that $\ell_0 := \deg L_0(x) \leq \ell_i := \deg L_i(x)$, for all $1 \leq i \leq r$. Define $\nu$ as the minimum multiplicity of the roots of $F(x) - \gamma_0$.

    If $F$ is an MVSP, then there exists a smallest positive integer $k$ such that $\nu \mid (p^k - 1)$ (in particular, $p \nmid \nu$), a positive integer $m$, polynomials $N_0, \ldots, N_r, A, B \in \mathbb{F}_q[x]$ with $L_i(x) \nmid N_i(x)$ for any $i = 0, \ldots, r$, and elements $\omega_0, \ldots, \omega_m \in \mathbb{F}_q$ with $\omega_0 \neq 0$ and $\omega_m = 1$ such that
    \begin{itemize}
        \item[(a)] $1 + \nu r = p^{mk}$;
        \item[(b)] $F(x) = L_0(x)^\nu N_0(x)^{p^{mk}} + \gamma_0$;
        \item[(c)] $L_0(x) = x \cdot A(x)^{p^{mk}} + B(x)^p$;
        \item[(d)] $F(x) = L_i(x)N_i(x)^p + \gamma_i$, $i = 1, \ldots, r$;
        \item[(e)] $\sum_{i=0}^m \omega_i L_0(x)^{p^{ki}} N_0(x)^{p^{mk}(p^{ki}-1)/\nu} = -\omega_0 (x^q - x) L_0'(x)$;
        \item[(f)] $\sum_{i=0}^m \omega_i (F(x) - \gamma_0)^{1+(p^{ki}-1)/\nu} = -\omega_0 (x^q - x) F'(x)$.
    \end{itemize}
    Note that items $(b)$ and $(d)$ imply that $F(x) - \gamma_i$ has a simple root in $\mathbb{F}_q$ for any $i = 1, \ldots, r$, and that the roots of $F(x) - \gamma_i$ which are not in $\mathbb{F}_q$ should be roots of $N_i(x)$, for $i = 0, \ldots, r$. 
    
    Conversely, if there exist $\omega_0, \ldots, \omega_m \in \mathbb{F}_q$ with $\omega_0 \neq 0$ and $\omega_m = 1$ such that item (f) holds, then $F$ is an MVSP.
\end{theorem}

In this work, we will use the previous results to classify all minimal value set binomials, that is, polynomials whose equation have two terms.\\

Let now \( p \) be a prime number, let \( \mathbb{F}_q \) be the finite field with \( q = p^h \) elements, where \( h \geq 1 \) is an integer, and let $t$ be a separating variable for the field extension $\mathbb{F}_q(\mathcal{C}) / \mathbb{F}_q$. There is a smallest positive integer $\nu$ such that
\[
\begin{vmatrix}
1 & x^q & y^q \\
1 & x & y \\
0 & D_t^{\nu} x & D_t^{\nu} y
\end{vmatrix}
\neq 0,
\]
where $D_t^{\nu}$ denotes the $\nu$-th Hasse derivative with respect to $t$. This integer is called the Frobenius order  of \( \mathcal{C} \). If \( \nu = 1 \), we say that \( \mathcal{C} \) is \textbf{\( \mathbb{F}_q \)-Frobenius classical}; otherwise, it is said to be \textbf{\( \mathbb{F}_q \)-Frobenius nonclassical}. By considering $t = x$, one can obtain the geometric interpretation of Frobenius nonclassical curves given before: the image of any nonsingular point of $\mathcal{C}$ under the $q$-Frobenius morphism lies on its tangent line.

% \begin{remark}
%     This determinant condition can be written as follows: let $f(x, y)$ be the defining polynomial of a planar curve $\mathcal{C}$. Then $\mathcal{C}$ is an $\mathbb{F}_q$-Frobenius nonclassical curve if, and only if, $f(x, y)$ divides
%     $$f_x \cdot (x^q - x) + f_y \cdot (y^q - y).$$
% \end{remark}

In this work, we will consider a quadrinomial curve \( \mathcal{C} \) of degree \( n \) defined over \( \mathbb{F}_q \), that is, \( \mathcal{C} \) is defined by an irreducible polynomial of the form
\[
F(x,y) = \sum_{\ell = 0}^{3} a_{i_\ell, j_\ell} x^{i_\ell} y^{j_\ell},
\]
where \( a_{i_\ell, j_\ell} \in \mathbb{F}_q^* \). There is no general criteria for the irreducibility of a quadrinomial curve. Proceeding as in \cite[Lemma 2.2]{nie}, we can obtain the following result.

\begin{proposition}
Let $\mathcal{C}$ be a quadrinomial curve (not necessarily irreducible) over a finite field $\mathbb{F}_q$. Then, up to a permutation of variables, if $x$, $y$ and $z$ do not divide the homogeneous equation of $\mathcal{C}$, then it can be written in one of the following forms:
$$k_1x^{m_1}z^{r_1} + k_2x^{m_2} + k_3y^{n_3} + z^{r_4} = 0, \quad k_1x^{m_1}z^{r_1} + k_2y^{n_2}z^{r_2} + k_3x^{m_3} + z^{r_4} = 0,$$
$$k_1x^{m_1}y^{n_1}z^{r_1} + k_2x^{m_2}y^{n_2} + k_3x^{m_3} + y^{n_4} = 0, \quad k_1x^{m_1}y^{n_1}z^{r_1} + k_2x^{m_2} + k_3y^{n_3} + z^{r_4} = 0,$$
$$k_1x^{m_1}z^{r_1} + k_2x^{m_2}z^{r_2} + k_3x^{m_3} + y^{n_4} = 0, \quad k_1x^{m_1}z^{r_1} + k_2y^{n_2}z^{r_2} + k_3x^{m_3} + y^{n_4} = 0,$$
$$k_1x^{m_1}y^{n_1}z^{r_1} + k_2x^{m_2}z^{r_2} + k_3y^{n_3} + z^{r_4} = 0, \quad k_1x^{m_1}z^{r_1} + k_2x^{m_2}z^{r_2} + k_3y^{n_3}z^{r_3} + x^{m_4} = 0,$$
$$k_1x^{m_1}z^{r_1} + k_2y^{n_2}z^{r_2} + k_3y^{n_3}z^{r_3} + x^{m_4} = 0, \quad k_1x^{m_1}y^{n_1} + k_2x^{m_2}z^{r_2} + k_3y^{n_3}z^{r_3} + z^{r_4} = 0,$$
$$k_1x^{m_1}y^{n_1}z^{r_1} + k_2x^{m_2}y^{n_2} + k_3x^{m_3}z^{r_3} + y^{n_4} = 0, \quad k_1x^{m_1}z^{r_1} + k_2x^{m_2}z^{r_2} + k_3y^{n_3}z^{r_3} + y^{n_4} = 0,$$
$$k_1x^{m_1}z^{r_1} + k_2x^{m_2}z^{r_2} + k_3y^{n_3}z^{r_3} + x^{m_4} = 0, \quad k_1x^{m_1}y^{n_1} + k_2x^{m_2}y^{n_2} + k_3x^{m_3}z^{r_3} + y^{n_4}z^{r_4} = 0,$$
$$k_1x^{m_1}y^{n_1}z^{r_1} + k_2x^{m_2}y^{n_2} + k_3x^{m_3}z^{r_3} + y^{n_4}z^{r_4} = 0, \quad k_1x^{m_1}y^{n_1}z^{r_1} + k_2x^{m_2}y^{n_2}z^{r_2} + k_3x^{m_3} + y^{n_4} = 0,$$
$$k_1x^{m_1}y^{n_1}z^{r_1} + k_2x^{m_2}y^{n_2}z^{r_2} + k_3y^{n_3}z^{r_3} + x^{m_4} = 0,$$
where $m_i, n_i, r_i$ are all positive integers for $i = 1,2,3,4$, $k_1k_2k_3 \neq 0$, and, if needed, $(m_1,r_1) \neq (m_2,r_2)$, $(n_2,r_2) \neq (n_3,r_3)$ and $(m_1,n_1,r_1) \neq (m_2,n_2,r_2)$.
\end{proposition}

\begin{proof}
  Suppose that the homogeneous equation of the curve $\mathcal{C}$ is  
$$
k_{1}x^{a_{1}}y^{b_{1}}z^{c_{1}} + k_{2}x^{a_{2}}y^{b_{2}}z^{c_{2}} + k_{3}x^{a_{3}}y^{b_{3}}z^{c_{3}} + k_{4}x^{a_{4}}y^{b_{4}}z^{c_{4}} = 0.
$$
We can consider two different cases depending on the permutation of the variables.

\medskip
\noindent
\textbf{1.} Assume that $a_{1} = \min\{a_{1},a_{2},a_{3}\}$, $b_{1} = \min\{b_{1},b_{2},b_{3}\}$, and $c_{2} = \min\{c_{1},c_{2},c_{3}\}$. Then  
\begin{equation*}
\begin{aligned}
    &k_{1}x^{a_{1}}y^{b_{1}}z^{c_{1}} + k_{2}x^{a_{2}}y^{b_{2}}z^{c_{2}} + k_{3}x^{a_{3}}y^{b_{3}}z^{c_{3}}  + k_{4}x^{a_{4}}y^{b_{4}}z^{c_{4}}\\
&= x^{a_{1}}y^{b_{1}}z^{c_{2}}\left( k_{1}z^{c_{1}-c_{2}} + k_{2}x^{a_{2}-a_{1}}y^{b_{2}-b_{1}} + k_{3}x^{a_{3}-a_{1}}y^{b_{3}-b_{1}}z^{c_{3}-c_{2}  } + k_{4}x^{a_{4}-a_{1}}y^{b_{4}-b_{1}}z^{c_{4}-c_{2}} \right).
\end{aligned}
\end{equation*}
\medskip
\noindent
\textbf{2.} Assume that $a_{1} = \min\{a_{1},a_{2},a_{3}\}$, $b_{2} = \min\{b_{1},b_{2},b_{3}\}$, and $c_{3} = \min\{c_{1},c_{2},c_{3}\}$. Then  
\begin{equation*}
\begin{aligned}
    &k_{1}x^{a_{1}}y^{b_{1}}z^{c_{1}} + k_{2}x^{a_{2}}y^{b_{2}}z^{c_{2}} + k_{3}x^{a_{3}}y^{b_{3}}z^{c_{3}} +
k_{4}x^{a_{4}}y^{b_{4}}z^{c_{4}}\\
&= x^{a_{1}}y^{b_{2}}z^{c_{3}}\left( k_{1}y^{b_{1}-b_{2}}z^{c_{1}-c_{3}} + k_{2}x^{a_{2}-a_{1}}z^{c_{2}-c_{3}} + k_{3}x^{a_{3}-a_{1}}y^{b_{3}-b_{2}} + k_4x^{a_{4}-a_{1}}y^{b_{4}-b_{2}}z^{c_{4}-c_{3}}\right).
\end{aligned}
\end{equation*}

Thus, the homogeneous equation of the curve $\mathcal{C}$ can be written either as  
$
k_{1}z^{c_{1}-c_{2}} + k_{2}x^{a_{2}-a_{1}}y^{b_{2}-b_{1}} + k_{3}x^{a_{3}-a_{1}}y^{b_{3}-b_{1}}z^{c_{3}-c_{2}  } + k_{4}x^{a_{4}-a_{1}}y^{b_{4}-b_{1}}z^{c_{4}-c_{2}} = 0
$
or  
$
 k_{1}y^{b_{1}-b_{2}}z^{c_{1}-c_{3}} + k_{2}x^{a_{2}-a_{1}}z^{c_{2}-c_{3}} + k_{3}x^{a_{3}-a_{1}}y^{b_{3}-b_{2}} + k_4x^{a_{4}-a_{1}}y^{b_{4}-b_{2}}z^{c_{4}-c_{3}} = 0$. By examining when the exponents in each term of these equations can be zero, we can split these two cases into $17$ reduced forms, as stated in the proposition, with respect to the permutation of variables.
\end{proof}

\begin{remark}\label{quadrinomials}
    After dehomogenizing the equations in the previous proposition, we observe that there are only five possible forms for the affine equation defining our curve: 
    \begin{itemize}
        \item[$\mathrm{(i)}$] $k_1x^{m_1} + k_2x^{m_2} + k_3y^{n_1} + y^{n_2} = 0, \quad m_1 \neq m_2,\ n_1 \neq n_2$;
        \item[$\mathrm{(ii)}$] $k_1x^{m_1} + k_2x^{m_2} + k_3y^{n_1} + 1 = 0, \quad m_1 \neq m_2$;
        \item[$\mathrm{(iii)}$] $k_1x^{m_1} + k_2x^{m_2} + k_3x^{m_3} + y^{n_4} = 0, \quad m_1 \neq m_2,\ m_3 \neq m_1, m_2$;
        \item[$\mathrm{(iv)}$] $k_1x^{m_1}y^{n_1} + k_2x^{m_2}y^{n_2} + k_3x^{m_3} + y^{n_3} = 0, \quad (m_1,n_1) \neq (m_2,n_2)$;
        \item[$\mathrm{(v)}$] $k_1x^{m_1}y^{n_1} + k_2x^{m_2} + k_3y^{n_2} + 1 = 0$.
    \end{itemize}
where  $m_i, n_i$ are all positive integers for $i = 1,2,3,4$, and $k_1k_2k_3 \neq 0$.
\end{remark}

Quadrinomial curves whose equations are of types $(i)$, $(ii)$ or $(iii)$ can be written in the form $\mathcal{C}: f(x) = g(y)$. They are then called \textbf{curves with separated variables}. As a consequence of our main theorem, we will classify such curves with respect to the $\mathbb{F}_{q}$-Frobenius nonclassicality. The relation between these concepts is given by the following result.

\begin{proposition}\cite[Corollary 3.5]{borges-separated}\label{borges-MVSPcondition}
    Let $f(x), g(x) \in \mathbb{F}_q[x]$ be nonconstant polynomials such that $f(x) - g(y) \notin \mathbb{F}_q[x^p, y^p]$. Suppose that the irreducible components of $\mathcal{C} : f(x) = g(y)$ are defined over $\mathbb{F}_q$. If these components are $\mathbb{F}_q$-Frobenius nonclassical, then $f(y)$ and $g(x)$ are MVSPs with $V_f = V_g$. Conversely, suppose that $f(x)$ and $g(y)$ are MVSPs with $V_f = V_g$. If $\#V_f > 2$ or $\#V_f = 2 = p$, then the irreducible components of $\mathcal{C}$ are $\mathbb{F}_q$-Frobenius nonclassical.
\end{proposition}

\section{Minimal Value Set Binomials}\label{section:MVSB}
In this section, we will classify the MVSPs with two terms. This classification will allow us to obtain the Frobenius nonclassical quadrinomial curves with separated variable of types $(i)$ and $(ii)$. As a combination of the results of this section, we will obtain the proof of Theorem A.

In order to prove this statement, we start from a binomial $F \in \mathbb{F}_q[x]$. Since the multiplication by a nonzero element of $\mathbb{F}_q$ does not change the cardinality of the value set, we can consider it monic and we have
\begin{equation}\label{eq:binomial-equation}
F(x) = x^a + \beta x^b, \quad a > b, \quad \beta \in \mathbb{F}_q^*.
\end{equation}
It is well-known that a monomial $f(x) = x^v$ is an MVSP if, and only if, $v \mid (q - 1)$ (see \cite[Theorem 8.3.29]{handbook}). Since $\# V_{x^a} = \# V_{x^a + \beta}$, we conclude that $F(x) = x^a + \beta$ is an MVSP if, and only if, $a \mid (q - 1)$. We can now assume $b > 0 $ in Equation \eqref{eq:binomial-equation}.

We already mentioned that an MVSP $F(x)$ with $\# V_F \leq 2$ does not fit in a general pattern, so we will treat them firstly.

\begin{proposition}\label{Vf=1}
    Let $F(x)$ be a binomial over $\mathbb{F}_q$. Then $F$ is an MVSP with $\# V_F = 1$ if, and only if,
    $$F(x) = x^{b + \ell(q - 1)} - x^{b},$$ for some $b, \ell \in \mathbb{N}$. Furthermore, the value set of this binomial is $V_F = \{0\}$.
\end{proposition}
\begin{proof}
      From Proposition \ref{prop:Vf-menor-que-2}, it is known that \( \# V_F = 1 \) if, and only if,
    $$
    x^a + \beta x^b = F(x) = (x^q - x)R(x) + \alpha,
    $$
    where \( R \in \mathbb{F}_q[x] \setminus \{0\} \) and \( \alpha \in \mathbb{F}_q \). Since $x^q = x$ for every $x \in \mathbb{F}_q$, we have $V_F = \{\alpha\}$.

    Clearly, in our case we must have \( \alpha = 0 \), and thus \( F(x) \) is divisible by \( x^q - x \). Performing the division algorithm, we obtain that \( a = b + \ell(q - 1) \) for some \( \ell \in \mathbb{N} \), and \( \beta = -1 \). Therefore,
    $$
    F(x) = x^{b + \ell(q - 1)} - x^b.
    $$
\end{proof}

\begin{proposition}\label{Vf=2}
    Let $F(x)$ be a binomial over $\mathbb{F}_q$. Then $F$ is an MVSP with $\# V_F = 2$ if, and only if, one of the following holds:
    \begin{enumerate}
        \item[$\mathrm{(i)}$] $F(x) = x^{q-1} + \beta x^{\frac{q-1}{2}}$, where $p > 2$ and $\beta^2 =1$. In this case, we have $V_F = \{0, 2\}$.
        \item[$\mathrm{(ii)}$] $F(x) = x^{\frac{2(q-1)}{3}} + \beta x^{\frac{q-1}{3}}$, where $p = 2$, $q > 2$ and $\beta^3 = 1$. In this case, we have $V_F = \{0, \beta^2\}$.
    \end{enumerate}
\end{proposition}

\begin{proof}
The verification that the polynomials in {\rm (i)} and {\rm (ii)} are MVSPs with the stated value set is straightforward.

Conversely, assume that $F(x) = x^a + \beta x^b$ is an MVSP over $\mathbb{F}_q$ with $\#V_F = 2$. From the definition of MVSPs, we have
\[
\frac{q-1}{2} < a \leq q-1,
\]
which implies $q > 2$. For $q = 3$, a direct computation shows that the only MVSP is $x^2 + \beta x$, with $\beta = \pm 1$.
For $q = 4$, the only MVSPs are of the form $x^2 + \beta x$, with $\beta \in \mathbb{F}_4^*$. Henceforth, we assume $q \ge 5$.

Let $\xi$ be a primitive element of $\mathbb{F}_q$. For all $k \ge 2$, the sequence $\{F(\xi^k)\}$ satisfies the recurrence relation
\begin{equation}\label{eqxi}
F(\xi^k) = (\xi^a + \xi^b)F(\xi^{k-1}) - \xi^{a+b}F(\xi^{k-2}).
\end{equation}

\medskip

\noindent
\textbf{(i): $p \neq 2$.} We need to consider some sub cases.

\noindent
\textbf{(i.1): The case $F(\xi) = 0$.}
In this case, $\xi^{a-b} = -\beta$. Since $a-b < q-1$ and $\xi$ is a primitive element, we have $\beta \neq -1$, and therefore $V_F = \{0, 1+\beta\}$. From Equation \eqref{eqxi} for $k=2$, we obtain $F(\xi^2)=1+\beta$ and $\xi^{a+b}=-1$. Thus,
\[
F(\xi^k) = (\xi^a+\xi^b)F(\xi^{k-1}) + F(\xi^{k-2}).
\]
In particular, $F(\xi^3) = (\xi^a + \xi^b)F(\xi^2)$ and we have $\xi^a + \xi^b$ equals $0$ or $1$.

If $\xi^a+\xi^b = 1$, then $F(\xi^3) = F(\xi^2)$, so $F(\xi^4)=2F(\xi^2)\notin V_F$, a contradiction. Hence, $\xi^a+\xi^b=0$, and this implies
\[
F(\xi^{2n}) = F(\xi^2), \qquad F(\xi^{2n+1}) = F(\xi) = 0,
\]
for all $n \ge 1$. Since \( x^{a - b} + \beta \) has exactly \( \gcd(a - b, q - 1) \) distinct roots in \( \mathbb{F}_q \), and \( a \leq q - 1 \), we can conclude that
$$
a = b + \frac{q - 1}{2}.
$$
From $F(\xi^2)=\xi^{2b}(1+\beta)$, we deduce $b = \frac{q-1}{2}$ and $\beta = 1$.

\smallskip

\noindent
\textbf{(i.2): The case $F(\xi) \neq 0$.}
If $\beta \neq -1$, then $F(\xi)=1+\beta$ and
\[
F(\xi^2) = (\xi^a+\xi^b-\xi^{a+b})F(\xi).
\]
Therefore, $\xi^a+\xi^b-\xi^{a+b}$ equals $0$ or $1$. In the first case, we obtain $F(\xi^2) = 0$, $F(\xi^3) = -\xi^{a + b} F(\xi)$ (so $\xi^{a + b} = -1$) and $F(\xi^4) = -F(\xi^3) \notin \{0, F(\xi)\}$, which is a contradiction. In the second case, we will obtain $F(\xi^k) = F(\xi) = 1 + \beta$ for all $k \geq 1$. Hence, the polynomial $F(x) - (1 + \beta)$ has at least $q - 1$ distinct roots in $\mathbb{F}_q$, which implies $a = q - 1$ and
  $$
  F(x) - (1 + \beta) = \prod_{i = 1}^{q - 1}(x - \xi^i) = x^{q - 1} - 1,
  $$
  which contradicts Equation \eqref{eq:binomial-equation}. Therefore, this case also does not occur. 
  
  So, if $F(\xi) \neq 0$, we must have $\beta = -1$. In this situation, necessarily $\xi^a+\xi^b=0$ (otherwise we would obtain $F(\xi^4) = - F(\xi) \notin V_f$), which yields, from Equation \eqref{eqxi}
\[
F(\xi^{2n})=0, \qquad F(\xi^{2n+1})=F(\xi),
\]
for all $n\ge1$. Arguing as before, we obtain
\[
a = b + \frac{q-1}{2}, \qquad b = \frac{q-1}{2}.
\]

\medskip

\noindent
\textbf{Case (ii): $p = 2$.} The proof is analogous to case (i). Therefore, we will omit some details.

\noindent
\textbf{(ii.1): The case $F(\xi) = 0$.} We conclude again that $\beta \neq 1$ and $\xi^{a+b} = 1$ (in particular, $a + b = q - 1$). Since the characteristic of the field is $2$, we have $\xi^a + \xi^b \neq 0$. From $F(\xi^3) = (\xi^a + \xi^b)(1 + \beta)$, we obtain $\xi^a + \xi^b = 1$ and
$$F(\xi^k) = F(\xi^{k - 1}) + F(\xi^{k - 2})$$
for all $k \geq 2$. This implies:
$$F(\xi^{3n + 1}) = 0 \quad \text{and} \quad F(\xi^{3n}) = 1 + \beta = F(\xi^{3n+2}) \quad \text{for all } n \geq 0.$$
Hence, the polynomial $F(x)$ now have $\frac{q-1}{3}$ roots in $\mathbb{F}_q$. We know that the number of its roots is $\gcd(a - b, q - 1)$ and that $a + b = q - 1$. We then obtain
$$a = \frac{2(q - 1)}{3}, \quad b = \frac{q - 1}{3}$$
and so $\beta = \xi^{a - b} = \xi^{(q-1)/3}$

\smallskip

\noindent
\textbf{(ii.2): The case $F(\xi) \neq 0$.} If $\beta = 1$, we analyze $F(\xi^2)$ and $F(\xi^3)$ to conclude that $\xi^a + \xi^b = 1$, $\xi^{a+b} = 1$ and $\xi^{3(a - b)} = 1$. We obtain again 
$$(a, b) = \left( \frac{2(q - 1)}{3},\, \frac{q - 1}{3} \right) \quad \textrm{and} \quad  \beta^3 = 1.$$

If $\beta \neq 1$, then $F(\xi^2) = (\xi^a + \xi^b + \xi^{a+b})F(\xi)$, and so $\xi^a + \xi^b + \xi^{a+b}$ equals $0$ or $1$. In the first case, we analyze $F(\xi^3)$, $F(\xi^4)$ and $F(\xi^5)$ to conclude that $\xi^{a+b} = \xi^{3(a - b)} = 1$. As before, we obtain
$$(a, b) = \left( \frac{2(q - 1)}{3},\, \frac{q - 1}{3} \right) \quad \textrm{and} \quad  \beta^3 = 1.$$
In the second case, we will obtain $F(\xi^k) = F(\xi)$ for all $k \geq 1$. Analogously to one of the sub cases of (i.2), we get a contradiction with Equation \eqref{eq:binomial-equation}.
\end{proof}

We can now study the binomials whose value sets have more than two elements. With notations as in Theorem \ref{MVSP-condition}, we first study the case $\gamma_0 = 0$.

\begin{lemma}\label{lem1}
    Let $F(x)$ be as in Equation \eqref{eq:binomial-equation} and suppose that $F$ is an MVSP. With notations as in Theorem \ref{MVSP-condition}, we have $\gamma_0 = 0$ if, and only if, $p \mid a$.
\end{lemma}
\begin{proof}
    Suppose first that $p \mid a$. By differentiating both sides of the identity in Theorem~\ref{MVSP-condition} (b), we obtain:
    \begin{equation*}
        \beta b x^{b - 1} = \nu L_0(x)^{\nu - 1} L_0'(x) N_0(x)^{p^{mk}}.
    \end{equation*}

   If \( b = 1 \), then \( \nu = 1 \) and \( N_0(x) \) is a constant polynomial. This implies that every root of \( F(x) - \gamma_0 \) lies in \( \mathbb{F}_q \) and has multiplicity \( \nu = 1 \). In particular, \( \ell_0 = a \). Hence, by the definition of \( \ell_0 \), we have \( \ell_i = a \) for all \( 0 \leq i \leq r \). Therefore, we may consider \( \gamma_0 = 0 \).

    If \( b > 1 \), then the terms \( N_0(x) \), \( L_0(x)^{\nu - 1} \) and \( L_0'(x) \) must be either a constant polynomial or a multiple of a positive power of \( x \). If either \( N_0(x) \) or \( L_0(x)^{\nu - 1} \) is a multiple of a positive power of \( x \), then $0$ is a root of $F(x) - \gamma_0$, and so \( \gamma_0 = 0 \). On the other hand, if both \( N_0(x) \) and \( L_0(x)^{\nu - 1} \) are constant polynomials, then \( \nu = 1 \) and the same argument used in the case \( b = 1 \) allows us to conclude that \( \gamma_0 = 0 \).\\

    Assume now that $\gamma_0 = 0$. The $\mathbb{F}_q$-roots of $F(x)$ are $0$ and the $(a - b)$-th roots of $-\beta$ in $\mathbb{F}_q$. Their multiplicities are $b$ and $p^c$ respectively, where $a - b = p^c \cdot u,\ p \nmid u$. Hence, since $p \nmid \nu$, we must have $\nu \in \{1, b\}$. Furthermore, $L_0(x)$ is either $x$ or of the form 
    $$x \cdot \prod_{\substack{\alpha^{a-b}+\beta=0 \\\alpha \in \mathbb{F}_q}}(x - \alpha).$$
    This yields three possibilities to consider: \\
    
    \noindent (a) \textbf{Case $\nu = b$ and $L_0(x)=x$:} Let $\alpha$ be a root of $N_0(x)$ (which must exist since $F(x)$, in this case, has roots outside of $\mathbb{F}_q$). Replacing $x = \alpha$ into Theorem \ref{MVSP-condition} (e) implies $\alpha = 0$, a contradiction. \\
    
    \noindent (b) \textbf{Case $\nu = b$ and $L_0(x)$ has non-zero roots:} Comparing the multiplicity of a root of $x^{a-b}+\beta$ in both sides of Theorem \ref{MVSP-condition} (b) implies $p^c = b+p^{mk}\tilde{u}$ for some $\tilde{u}$. This forces $p \mid b$. Since $\nu = b$, this implies $p \mid \nu$, contradicting the condition $\nu \mid (p^k - 1)$ from Theorem \ref{MVSP-condition}. \\
    
    \noindent (c) \textbf{Case $\nu = 1$:} This is the only surviving branch and, from the cases above, we can assume that $1 = \nu \neq b$. The information on $\nu$ implies that $k = 1$ (from the definition of $k$) and that $F(x)$ has simple roots (from the definition of $\nu$). Since $0$ is a root of $F(x)$ with multiplicity $b > 1$, we then conclude that the roots of $x^{a - b} + \beta$ are simple, $L_0(x) = x^{a - b + 1} + \beta x$ and $N_0(x)^{p^m} = x^{b - 1}$. In particular, $b \equiv 1 \pmod{p}$. From Theorem \ref{MVSP-condition} (c), we conclude that $L_0'(x) = (a - b + 1)x^{a - b} + \beta$ is a $p^m$-power. This is only possible if $a - b + 1 \equiv 0 \pmod{p}$ (since the alternative, $p \mid a - b$, would contradict the simplicity of the roots of $x^{a - b} + \beta$). This condition, combined with $b \equiv 1 \pmod{p}$,
    immediately yields $p \mid a$.
\end{proof}

\begin{remark}\label{rmk}
    The proof of the previous lemma also implies that, if $\gamma_0 = 0$, then $b \equiv 1 \pmod{p}$ and $\nu = 1$. In particular, $F(x)$ has all its roots in $\mathbb{F}_q$ and they are all simple (except possibly the root $0$, which has multiplicity $b$). Furthermore, $L_0(x) = x^{a - b + 1} + \beta x$ and $N_0(x)^{p^m} = x^{b - 1}$.
\end{remark}

\begin{proposition}\label{Vf>2}
    Let $F(x)$ be as in Equation \eqref{eq:binomial-equation} and suppose that $F$ is an MVSP with $\gamma_0 = 0$. Then there exist integers $s \geq 2$ and $\ell > 0$ such that $s\ell$ divides $n$, where $q = p^n$, and 
    $$F(x) = x^{p^{\ell}t} + \beta x^t, \quad t = \frac{q - 1}{p^{s\ell} - 1}, \quad \beta \in \mathbb{F}_{p^{s\ell}}^* \quad \textrm{and}\quad N_{\mathbb{F}_{p^{s\ell}} / \mathbb{F}_{p^{\ell}}}(\beta) = (-1)^s.$$
    Furthermore, the value set of this polynomial consists of the $p^{(s-1)\ell}$ roots of the polynomial $\sum_{i = 0}^{(s - 1)\ell}\frac{(-1)^{i + 1}}{\beta^{p^{i\ell} + \cdots + 1}}x^{p^{i\ell}}$.
\end{proposition}

\begin{proof}
    From Remark~\ref{rmk}, we know that $\nu = 1$ and $b \equiv 1 \pmod{p}$. Theorem~\ref{MVSP-condition} (f) implies the existence of $\omega_0, \ldots, \omega_m \in \mathbb{F}_q$ with $\omega_0 \neq 0$ and $\omega_m = 1$ such that
    \begin{equation}\label{eq:1}
    \sum_{i=0}^{m} \omega_i (x^a + \beta x^b)^{p^i} = -\omega_0 \beta(x^{q - 1 + b} - x^b).
    \end{equation}
    Comparing the terms of highest  degree on both sides of Equation~\eqref{eq:1}, we obtain
    \begin{equation}\label{eq:2}
    \quad \omega_0 = -\frac{1}{\beta}, \quad \text{and} \quad a p^m = q  - 1 + b.
    \end{equation}

    Since the term $\omega_0 x^a$ appears on the left-hand side of Equation~\eqref{eq:1}, but does not appear on the right-hand side, it must be canceled out. Therefore, there exists $\ell > 0$ such that
    $$\omega_0 x^a = -\omega_{\ell} \beta^{p^{\ell}} x^{bp^{\ell}},$$
    and $\omega_i = 0$ for all $0 < i < \ell$, that is, $a = bp^{\ell}$ and $\omega_{\ell} = 1 / \beta^{1 + p^{\ell}}$. From Equation \eqref{eq:2}, we then conclude that $b = \frac{q - 1}{p^{\ell + m} - 1}$.

    If $\ell = m$, then we can easily check that this gives the desired result (with $s = 2$). Otherwise, we have a recursive cancellation that works in the same manner for all highest-degree terms; that is, the term
$
\omega_{k\ell} x^{a p^{k\ell}}
$
is cancelled only by the term
$
\omega_{(k+1)\ell}\,\beta^{ p^{(k+1)\ell}}\, x^{b p^{(k+1)\ell}},
$
and thus there exist integers $s > 2$ and $\ell > 0$ such that $(s - 1)\ell = m$, $\omega_i = 0$ for all 
    $$i \in \{1, \dots, p^{(s - 1)\ell} - 1\} \setminus \{p^{\ell}, p^{2\ell}, \dots, p^{(s - 2)\ell}\},$$ 
    $\omega_{j\ell} = (-1)^{j + 1} / \beta^{1 + p^\ell + \cdots + p^{j\ell}}$ for all $j \in \{1, \dots, s - 2\}$, and
    \begin{equation}\label{eq:3}
        (-1)^s = \beta^{1 + p^\ell + \cdots + p^{(s - 1)\ell}}.
    \end{equation}

    From Equation~\eqref{eq:2}, we have that $b = \dfrac{q - 1}{p^{s\ell} - 1} =: t$, and thus $s\ell \mid n$. Since $a - b = (p^{\ell} - 1)t$ and $x^{(p^{\ell} - 1)t} + \beta$ has a root in $\mathbb{F}_q$, we conclude that $\beta \in \mathbb{F}_{p^{s\ell}}$. Finally, from Equation~\eqref{eq:3} we obtain
    $$N_{\mathbb{F}_{p^{s\ell}} / \mathbb{F}_{p^{\ell}}}(\beta) = (-1)^s.$$
    
    The information on the value set of $F(x)$ comes from Equation \eqref{eq:1}, by observing that every element of $\mathbb{F}_q$ is a root of the right-hand side of this equation. We then conclude that
    $$V_{x^{p^{\ell}t}+\beta x^t} \subseteq Z\left(\sum_{i = 0}^{(s - 1)\ell}\frac{(-1)^{i + 1}}{\beta^{p^{i\ell} + \cdots + 1}}x^{p^{i\ell}}\right).$$
    The equality of both sets comes from the fact that they have the same cardinality.
    
    Direct computations with the converse of Theorem \ref{MVSP-condition} and the $\omega_i$'s we just obtained show that this polynomial is indeed an MVSP. 
\end{proof}

Finally, we consider the case $\gamma_0 \neq 0$.

\begin{proposition}\label{Vf>22}
    Let $F(x)$ be as in Equation \eqref{eq:binomial-equation} and suppose that $F$ is an MVSP with $\gamma_0 \neq 0$. Then there exists an integer $m > 0$ such that $m \mid n$, where $q = p^n$ and $p \neq 2$, and
    $$
    F(x) = x^{2t} + \beta x^t,
    $$
    with $t = \frac{q - 1}{p^m - 1}$ and $\beta \in \mathbb{F}_{p^m}^*$. Furthermore, the value set of this polynomial is $V_F = \{u^2 - \frac{\beta^2}{4} : u \in \mathbb{F}_{p^m}\}$, which is a set of cardinality $\frac{p^m-1}{2}+1$.
\end{proposition}

\begin{proof}
    Suppose,  w.l.o.g., that $\gamma_1 = 0$. Since $\ell_1 > 1$ (otherwise we could consider $\gamma_0 = 0$), we know that $\ell_1 = 1 + \gcd(a - b, q - 1)$. Thus, $\ell_0 \leq \gcd(a - b, q - 1)$. Given that 
    $$
    q = \sum_{i = 0}^r \ell_i \leq 1 + 2\gcd(a - b, q - 1) + (r - 1)a
    $$
    and $r \leq \frac{q - 1}{a}$, we get $a \leq 2\gcd(a - b, q - 1)$. Therefore, $\gcd(a - b, q - 1) > \frac{a - b}{2}$, which implies that $\gcd(a - b, q - 1) = a - b$, and hence $(a - b) \mid (q - 1)$. Equality $q = \sum \ell_i$ can now be written as
    $$q = \ell_0 + [1 + (a - b)] + \sum_{i=2}^{r} \ell_i \leq \ell_0 + [1 + (a - b)] + (r - 1)a \leq \ell_0 + [1 + (a - b)] + (q - a -1),$$
    and this implies $\ell_0 \geq b$.

   Lemma \ref{lem1} states that \( p \nmid a \). Then, from Theorem~\ref{MVSP-condition} (f), we have:

    \begin{equation}\label{eq:11}
        \sum_{i = 0}^{m} \omega_i (x^a + \beta x^b - \gamma_0)^{\frac{p^{ki} - 1}{\nu} + 1} 
        = -\omega_0 (a x^{q - 1 + a} + b \beta x^{q - 1 + b} - a x^a - b \beta x^b),
    \end{equation}
    with $\omega_0 \ne 0$ and $\omega_m = 1$. By comparing the terms of highest degree on both sides of Equation~\eqref{eq:11}, we get $a(r + 1) = q - 1 + a$, so $ar = q - 1$, and therefore $a \mid (q - 1)$. 

    Moreover, since $q - 1 + b = ar + b$, the term $-\omega_0 b \beta x^{q - 1 + b}$ corresponds on the left-hand side of the previous equation to $(r + 1)\beta x^{ar + b}$. Since $1 + \nu r = p^{mk}$, we obtain
    $$
        p \mid b \;\Leftrightarrow\; r \equiv -1 \pmod{p} \;\Leftrightarrow\; \nu \equiv 1 \pmod{p}.
    $$

    However, the coefficient of $x^b$ on the left-hand side of Equation~\eqref{eq:11} is
    $$
    \omega_0 \beta + \sum_{i = 1}^{m} \omega_i \beta \left(\frac{p^{ki} - 1}{\nu} + 1\right)(-\gamma_0)^{\frac{p^{ki} - 1}{\nu}},
    $$
    and on the right-hand side it is $\omega_0 b \beta$. So if $\nu \equiv 1 \pmod{p}$, then we would have $\omega_0 \beta = 0$, which is a contradiction. Hence, $p \nmid b$.

    Since $p \nmid r$ and $r \not\equiv -1 \pmod{p}$, we conclude that $p \ne 2$, and the term $\frac{r(r + 1)}{2} \beta^2 x^{a(r - 1) + 2b}$ appears in the expansion of the left-hand side of Equation~\eqref{eq:11}. Since  $1-b/a < 1< p^{(m-1)k}/2$ for $m \neq 1$ (the case $m=1$ is trivial from Equation \eqref{eq:11}) and $\nu \leq p^k-1$ , it follows that $\nu (2 - 2b/a) < p^{(m-1)k} (p^k -1)$. Moreover, since $\nu r=p^{mk}-1$ we have

    $$
    a\left( \frac{p^{(m - 1)k} - 1}{\nu} + 1 \right) < a(r - 1) + 2b,
    $$
    and therefore the cancellation of the term $\frac{r(r + 1)}{2} \beta^2 x^{a(r - 1) + 2b}$ must occur in the expansion of $(x^a + \beta x^b - \gamma_0)^{r + 1}$. 
    
    The terms of this expansion are of the form $\binom{r + 1}{r + 1 - s - j, s, j}\beta^s(-\gamma_0)^j x^{a(r + 1 - s - j) + bs}$. We have $a(r + 1 - s - j) + bs = a(r - 1) + 2b$ if, and only if, $(s, j) \in \{(2, 0), (0,1)\}$. Therefore,

    $$
    \frac{r}{2} \beta^2 = \gamma_0 \quad \textrm{and} \quad a = 2b.
    $$
    Since $b \leq \ell_0 \leq a - b$, we get $\ell_0 = b$. Also, since $p \nmid b$, we have $b = \ell_0 \equiv 1 \pmod{p}$ (from Theorem \ref{MVSP-condition}(c)), hence $a \equiv 2 \pmod{p}$, and we obtain $r \equiv -\frac{1}{2} \pmod{p}$. Therefore, $\gamma_0 = -\beta^2 / 4$.

    Thus, $F(x) - \gamma_0 = \left(x^b + \frac{\beta}{2}\right)^2$, and we conclude that $\nu = 2$ and $k = 1$. So
    $$
    \frac{q - 1}{a} = r = \frac{p^m - 1}{2}, \quad \text{that is}, \quad 
    a = 2 \cdot \frac{q - 1}{p^m - 1}, \quad \text{and} \quad b = \frac{q - 1}{p^m - 1} =: t.
    $$
    As in the previous proposition, we conclude that $\beta \in \mathbb{F}_{p^m}^*$.

    Note that the function $x \mapsto x^t$ is the norm function from $\mathbb{F}_q$ to $\mathbb{F}_{p^m}$. Therefore, we have $V_{x^{2t} + \beta x^t} = \{X^2 + \beta X : X \in \mathbb{F}_{p^m}\}$. Direct computations show that $X^2 + \beta X$ is an MVSP over $\mathbb{F}_{p^m}$. This will imply that $x^{2t} + \beta x^t$ is an MVSP over $\mathbb{F}_q$ since $\frac{q - 1}{2t} = \frac{p^m - 1}{2}$.
\end{proof}

\section{Quadrinomial curves of type (i)}\label{section:quadrinomial1}

 As an application of our classification of minimal value set binomials, and considering Proposition \ref{borges-MVSPcondition}, we will determine the Frobenius (non-)classicality of some quadrinomial curves. With notations as in Remark \ref{quadrinomials}, in this section we study those of type (i), that is, curves defined by the plane equation
$$
\mathcal{C}: y^c + \gamma y^d = \alpha x^a + \beta x^b.
$$

Let $f(x) = \alpha x^a + \beta x^b$ and $g(y) = y^c + \gamma y^d$. A necessary condition for the irreducible components of $\mathcal{C}$ to be Frobenius nonclassical curves is that $f$ and $g$ are MVSPs with the same value set. Moreover, if $\# V_f > 2$, this is also a sufficient condition. Therefore, we divide the analysis into cases according to whether $\# V_f = 1$, $\# V_f = 2$, or $\# V_f > 2$. 

At the end of this section, we will prove that, except for the case $\# V_f = 1$, a quadrinomial curve of type (i) whose irreducible components are Frobenius nonclassical can be written, after a suitable $\mathbb{F}_q$-projectivity, as $f(x) = f(y)$. In particular, such quadrinomial curves are never irreducible, since the line $\ell: x - y = 0$ is always a component of $\mathcal{C}: f(x) - f(y) = 0$. We start with the case $\# V_f = 1$.

\begin{proposition}
Let $\mathcal{C}: g(y) = f(x)$ be a quadrinomial curve of type (i) such that $\# V_f = 1$. Suppose that the irreducible components of $\mathcal{C}$ are defined over $\mathbb{F}_q$. Then these components are $\mathbb{F}_q$-Frobenius nonclassical if, and only if, 
$$
\mathcal{C}: y^{d + q - 1} - y^d = \alpha x^{b + q - 1} - \alpha x^b,
$$
where $\alpha \in \mathbb{F}_q^*$ and $b \equiv d \equiv 1$ (mod $p$).
\end{proposition}

\begin{proof}
Since $f$ and $g$ must be MVSPs with the same value set, by Proposition~\ref{Vf=1} we have $f(x) = \alpha x^{b + j(q - 1)} - \alpha x^b$ with $\alpha \in \mathbb{F}_q^*$, and $g(y) = y^{d + j'(q - 1)} - y^d$. As mentioned in the beggining of Section \ref{section:background}, we have $f(x) \notin \mathbb{F}_q[x^p]$, otherwise it would not be an MVSP. Analogously, $g(y) \notin \mathbb{F}_q[y^p]$. Under the hypothesis that the irreducible components of $\mathcal{C}: f(x) - g(y)$ are defined over $\mathbb{F}_q$, we can conclude, from \cite[Theorem 3.4]{borges-separated}, that these components are $\mathbb{F}_q$-Frobenius nonclassical if, and only if, there exist a monic polynomial $T(x) \in \mathbb{F}_q[x]$ and $\theta\in\mathbb{F}_q^*$ such that
\begin{equation} \label{eqq1}
    T(f(x)) = \theta(\alpha (b - j) x^{b + (j+1)(q-1)} - \alpha b x^{b + q -1} - \alpha(b - j) x^{b + j (q-1)} + \alpha b x^b),
\end{equation}
and
\begin{equation} \label{eqq2}
    T(g(y)) = \theta((d - j') y^{d + (j'+1)(q-1)} - d y^{d + q -1} - (d - j') y^{d + j'(q-1)} + d y^d).
\end{equation}

Denote by $D$ the degree of $T$. Now, assume that $b - j \not\equiv 0 \pmod{p}$. Then, comparing the degrees of the polynomials in~\eqref{eqq1}, we have
$$
(b + j(q-1))D= b + j(q-1) + q-1 \Longrightarrow(b + j(q - 1))(D - 1) = q - 1.
$$
This equality is never attained. Thus, we conclude that $b - j \equiv 0 \pmod{p}$. Hence, comparing the terms with the highest degrees in~\eqref{eqq1} again, we obtain $j = D = 1$, $T(x) = x$ and $\theta b = -1$. Now, from~\eqref{eqq2}, this implies that $1 = j' \equiv d \equiv b \pmod{p}$, as desired.
\end{proof}

\begin{remark}
    We could not provide a general criteria for the irreducibility of this curve, so in Theorem B we will only consider ``suitable'' values of $b$ and $d$ such that the curve is irreducible.

    Suitable values do exist. Choose, for example, $b$ and $d$ such that $\gcd(b+q-1,d+q-1)=1$. The irreducibility (in fact, the absolute irreducibility) of the curve now comes from \cite[Theorem III.1.B]{schmidt}. 
\end{remark}

\begin{proposition}
    Let $\mathcal{C} \colon g(y) = f(x)$ be a quadrinomial curve of type (i) such that $\# V_f = 2$ and whose irreducible components are defined over $\mathbb{F}_q$. Then these irreducible components are $\mathbb{F}_q$-Frobenius nonclassical if, and only if, one of the following holds:

    \begin{itemize}
        \item[(i)]  $\mathbb{F}_4 \subseteq \mathbb{F}_q$ and $\mathcal{C}: y^{\frac{2(q-1)}{3}} + y^{\frac{q-1}{3}} = x^{\frac{2(q-1)}{3}} + x^{\frac{q-1}{3}}$;
        \item[(ii)] $p = 3$ and $\mathcal{C}: y^{q-1} + y^{\frac{q-1}{2}} = x^{q-1} + x^{\frac{q-1}{2}}$.
    \end{itemize}
\end{proposition}

\begin{proof}
First, suppose that $p = 2$. Then, by Proposition~\ref{borges-MVSPcondition} and Proposition~\ref{Vf=2}, the irreducible components of the curve $\mathcal{C}$ are $\mathbb{F}_q$-Frobenius nonclassical if, and only if,
$$
f(x) = \alpha x^{\frac{2(q - 1)}{3}} + \alpha\beta x^{\frac{q - 1}{3}} \quad \text{and} \quad g(y) = y^{\frac{2(q - 1)}{3}} + \beta' y^{\frac{q - 1}{3}},
$$
where $\alpha \in \mathbb{F}_q^*, \beta^3 = 1 = \beta'^3$, and $ \{0, \alpha \beta^2\} = V_f = V_g = \{0, \beta'^2\}$. In particular, $\alpha = \beta'^2/\beta^2= \beta/\beta'$. We then obtain
$$\mathcal{C}: \beta'y^{\frac{2(q-1)}{3}} + \beta'^2 y^{\frac{q-1}{3}} = \beta x^{\frac{2(q-1)}{3}} + \beta^2 x^{\frac{q-1}{3}}.$$
Since $\frac{1}{\beta^2}$ and $\frac{1}{\beta'^2}$ are in the image of the norm function $x \mapsto x^{(q-1)/3}$, we can apply a transformation $(x^{(q-1)/3}, y^{(q-1)/3}) \mapsto \left(\frac{x^{(q-1)/3}}{\beta^2}, \frac{y^{(q-1)/3}}{\beta'^2}\right)$ to put $\mathcal{C}$ in the desired form. Note that
$$y^{\frac{2(q-1)}{3}}+y^{\frac{q-1}{3}}+x^{\frac{2(q-1)}{3}}+x^{\frac{q-1}{3}} = \left(y^{\frac{q-1}{3}} + x^{\frac{q-1}{3}}\right) \cdot \left(y^{\frac{q-1}{3}} + x^{\frac{q-1}{3}} + 1\right).$$
From \cite[Corollary 2.10]{beelen-pelikaan}, the component $y^{(q-1)/3}+x^{(q-1)/3}+1$ is irreducible. From \cite[Corollary 2.5]{beelen-pelikaan}, the other component can be written as the product of lines
$$y^{\frac{q-1}{3}} + x^{\frac{q-1}{3}} = \prod_{\xi^{(q-1)/3}=1}(y + \xi x).$$
We have $\{\xi^{(q-1)/3} = 1\} \subseteq \mathbb{F}_q$, so the lines $y + \xi x$ are all defined over $\mathbb{F}_q$. Therefore, the quadrinomial curve is the product of a Frobenius nonclassical Fermat curve and $\frac{q-1}{3}$ lines.

Now, suppose $p > 2$. If $f$ and $g$ are MVSPs, then they have the same form as the polynomial $F$ in Proposition~\ref{Vf=2} (i). In the case $f(x) = x^{q-1} - x^{(q-1)/2}$, we apply the transformation $x \mapsto \zeta x$, where $\zeta$ is a primitive element of $\mathbb{F}_q$, to put the curve in the form $\mathcal{C}: y^{q-1}+y^{\frac{q-1}{2}}-x^{q-1}-x^{\frac{q-1}{2}} = 0$. Note that
$$y^{q-1}+y^{\frac{q-1}{2}}-x^{q-1}-x^{\frac{q-1}{2}} = \left(y^{\frac{q-1}{2}} - x^{\frac{q-1}{2}}\right) \cdot \left(y^{\frac{q-1}{2}} + x^{\frac{q-1}{2}} + 1\right),$$
so the quadrinomial curve is the product of a Fermat curve and $\frac{q-1}{2}$ lines $y - \xi x$, with $\xi^{(q-1)/2}=1$. From \cite[Theorem 2]{garcia1}, the Fermat curve is Frobenius nonclassical if, and only if, $p = 3$.
\end{proof}

Since the curves we just obtained are reducible, they are not included in the classification of irreducible curves in Theorem B.

\begin{proposition}
    Let $\mathcal{C} \colon g(y) = f(x)$ be a quadrinomial curve of type (i) such that $\# V_f > 2$, and suppose that its irreducible components are defined over $\mathbb{F}_q$. Then the irreducible components of $\mathcal{C}$ are $\mathbb{F}_q$-Frobenius nonclassical if, and only if, after a suitable $\mathbb{F}_q$-projectivity $\mathcal{C}$ is defined by one of the following plane equations:
    \begin{itemize}
        \item[(i)] $y^{p^{\ell}t} + \beta y^{t} = x^{p^{\ell}t} + \beta x^t$, where $t = \frac{q - 1}{p^{s\ell} - 1}, \beta \in \mathbb{F}_{p^{s\ell}}^*$ and $N_{\mathbb{F}_{p^{s\ell}}/\mathbb{F}_{p^{\ell}}}(\beta) = (-1)^s$ with $s \geq 2$;
        
        \item[(ii)] $y^{2t} + \beta y^t = x^{2t} + \beta x^t$, where $p > 2$, $t = \frac{q - 1}{p^m - 1}$ and $\beta \in \mathbb{F}_{p^m}^*$.
    \end{itemize}
Note that both curves are reducible, since the line $x - y$ is a component of them. Therefore, they are not included in the classification of irreducible curves in Theorem B.
\end{proposition}

\begin{proof}
  If $f$ and $g$ are MVSPs, then, by analyzing the possibilities in Propositions~\ref{Vf>2} and~\ref{Vf>22} under the condition that $\# V_f = \# V_g$, we see that there are two possible equations for the curve $\mathcal{C}$:
    \begin{itemize}
        \item $y^{p^{\ell'}t'} + \beta' y^{t'} = \alpha x^{p^{\ell}t} + \alpha \beta x^t$, where $t = \frac{q - 1}{p^{s\ell} - 1}$, $t' = \frac{q - 1}{p^{s'\ell'} - 1}, \alpha \in \mathbb{F}_q^*,$ $\beta \in \mathbb{F}_{p^{s\ell}}$, $\beta' \in \mathbb{F}_{p^{s'\ell'}}$, and $N_{\mathbb{F}_{p^{s\ell}}/\mathbb{F}_{p^{\ell}}}(\beta) = (-1)^s$, $N_{\mathbb{F}_{p^{s'\ell'}}/\mathbb{F}_{p^{\ell'}}}(\beta') = (-1)^{s'}$ with $s, s' \geq 2$;

        \item $y^{2t} + \beta y^{t} = \alpha x^{2t}  + \alpha \beta' x^t$, where $p > 2$, $t = \frac{q - 1}{p^m - 1}$ and $\beta, \beta' \in \mathbb{F}_{p^{m}}^*$.
    \end{itemize}

    Let us analyze the first case. We already showed that
    $$V_{y^{p^{\ell}t}+\beta y^t} = Z\left(\sum_{i = 0}^{(s - 1)\ell}\frac{(-1)^{i + 1}}{\beta^{p^{i\ell} + \cdots + 1}}x^{p^{i\ell}}\right).$$
    In particular, it is contained in $\mathbb{F}_{p^{s\ell}}$. Given that $V_f = V_g$, we must have $\alpha \in \mathbb{F}_{p^{s\ell}}$. Therefore, $\frac{1}{\alpha^{p^{\ell(s - 1)}}}$ is in the image of the norm function $x \mapsto x^t$, so $\frac{1}{\alpha^{p^{\ell(s - 1)}}} = \gamma^t$ for some $\gamma \in \mathbb{F}_q$ and we can apply the changing of coordinates $x \mapsto \gamma x$ and consider $\alpha = 1$.
    
    The polynomial $f(x)$ is now given by $x^{p^{\ell}t} + \tilde{\beta}x^t$, with $\tilde{\beta} = \alpha \beta \gamma^t$. The value set of $f$ is now given by
    $$V_f = Z\left(\sum_{i = 0}^{(s - 1)\ell}\frac{(-1)^{i + 1}}{\tilde{\beta}^{p^{i\ell} + \cdots + 1}}x^{p^{i\ell}}\right).$$
    Since $V_f = V_g$, we have
    $$Z\left(\sum_{i = 0}^{(s' - 1)\ell'}\frac{(-1)^{i + 1}}{\tilde{\beta}^{p^{i\ell'} + \cdots + 1}}x^{p^{i\ell'}}\right) = Z\left(\sum_{i = 0}^{(s - 1)\ell}\frac{(-1)^{i + 1}}{\beta^{p^{i\ell} + \cdots + 1}}x^{p^{i\ell}}\right).$$
    These polynomials have the same degree and the same roots, so they must be the same. This implies $\ell = \ell', s = s'$ and $\beta = \tilde{\beta}$.

    Now, consider the second case. As before, we must have $\alpha \in \mathbb{F}_{p^m}$. We analyze, as in the proof of the case $\# V_f = 1$, the existence of a monic polynomial $T$ and $\theta$ such that
    \begin{equation}\label{eq-y}
        T(y^{2t}+\beta y^t) = \theta(y^q - y)(2y^{2t - 1} + \beta y^{t - 1})
    \end{equation}
    and 
    $$T(\alpha x^{2t} + \alpha\beta' x^t) = \theta(x^q - x)(2\alpha x^{2t-1} + \alpha\beta' x^{t-1}).$$
    By analyzing the terms with the highest and the two lowest degrees in Equation \eqref{eq-y}, we conclude that $\theta = 1/2$ and
    $$T(z) = z^{(p^m+1)/2} + \cdots -\frac{z}{2}.$$
    
    By applying $T$ in $f(x)$, we conclude that $\alpha^{(p^m+1)/2} = \alpha$. This implies that $\alpha^{(p^m - 1)/2} = 1$ and, in particular, $\alpha$ is a square in $\mathbb{F}_{p^m}$. As before, we can rescale $x$ and assume $\alpha = 1$ and $f(x) = x^{2t} + \tilde\beta x^t$. Applying $T$ to this polynomial implies $\beta^2 = \tilde\beta^2$. Therefore, $\tilde\beta = \pm \beta$. In the case $\tilde\beta = - \beta$, we can rescale $x$ again to obtain the desired equation for $\mathcal{C}$. Note that
    $$y^{2t}+\beta y^t - x^{2t} - \beta x^t = (y^t + x^t + \beta)(y^t - x^t),$$
    and the quadrinomial curve is the product of a Frobenius nonclassical Fermat curve and $t$ lines defined over $\mathbb{F}_q$.
\end{proof}

\section{Quadrinomial curves of type (ii)}\label{section:quadrinomial2}
In this section, we study quadrinomial curves of type (ii), that is, curves defined by the plane equation
$$
y^a = \alpha x^b + \beta x^c + \gamma, \quad \alpha\beta\gamma \neq 0.
$$
For those curves, we will apply the following result, as well as its corollary (which will help us during some computations).

\begin{proposition}\cite[Theorem 5.3]{borges-separated}\label{borges-condition}
    Let $f\in\mathbb{F}_q[x]$ be a nonconstant polynomial and $m\geq1$ be an integer such that $y^m-f(x)\notin\mathbb{F}_q[x^p,y^p]$. Then, the irreducible components of $\mathcal{C}: y^m-f(x)$ are $\mathbb{F}_q$-Frobenius nonclassical if, and only if, $m\mid q-1$ and $$m\cdot f(x)(f(x)^{(q-1)/m}-1)=(x^{q}-x)f'(x).$$
\end{proposition}

\begin{corollary}\cite[Corollary 5.4]{borges-separated}\label{borges-corollary}
    Consider notation and hypothesis as in the previous theorem. If the components of $\mathcal{C}: y^m-f(x)$ are $\mathbb{F}_{q}$-Frobenius nonclassical, then
    \begin{enumerate}
        \item[$\mathrm{(i)}$] $p \nmid m$ and $f'(x) \neq 0$;
        \item[$\mathrm{(ii)}$] $\frac{mq}{m + q - 1} \leq \deg f \leq m$, and the upper bound (resp. lower bound) for $\deg f$ is attained if, and only if, $p \nmid \deg f$ (resp. $f'$ is constant);
        \item[$\mathrm{(iii)}$] if $f = c \prod_{i = 1}^s(x - a_i)^{k_i} \in \overline{\mathbb{F}_q}[x]$, then $a_i \in \mathbb{F}_{q}$ if, and only if, $p \nmid k_i$. In particular, all simple roots of $f$ lie in $\mathbb{F}_{q}$;
        \item[$\mathrm{(iv)}$] $f$ has an $\mathbb{F}_{q}$-root, and if $f$ has a simple root, then $m \equiv 1$ (mod $p$);
        \item[$\mathrm{(v)}$] $m \equiv 1$ (mod $p$) if, and only if, $f'' = 0$;
        \item[$\mathrm{(vi)}$] if an $\mathbb{F}_{q}$-root of $f$ has multiplicity $k$, with $0 < k < \deg f$, then $k \leq \frac{m - 1}{\# V_f}$ and $k \equiv m$ (mod $p$).
    \end{enumerate}
\end{corollary}

The following proposition will be very important to prove the irreducibility of some curves of this section.

\begin{proposition}\cite[Lemma I.2C]{schmidt}\label{irreducible}
    The curve given by $\mathcal{C}: y^d = f(x)$ over $\mathbb{F}_q$ is absolutely irreducible if, and only if, $\gcd(d, d_1, \ldots, d_r) = 1$, where 
    $$f(x) = a(x - x_1)^{d_1} \cdots (x-x_r)^{d_r}$$
    is the factorization of $f(x)$ in $\overline{\mathbb{F}_q}$.
\end{proposition}

As in the previous section, we divide our analysis into cases. Throughout this section, we assume $f(x) = \alpha x^b + \beta x^c + \gamma$ and $g(y) = y^a$. Since $\# V_g = 1 + \frac{q - 1}{a}$, there are no $\mathbb{F}_q$-Frobenius nonclassical quadrinomial curves of type (ii) such that $\# V_f = 1$.

\begin{proposition}
Let $\mathcal{C} \colon g(y) = f(x)$ be a quadrinomial curve of type (ii) such that $\# V_f = 2$ and whose irreducible components are defined over $\mathbb{F}_q$. Then these components are $\mathbb{F}_q$-Frobenius nonclassical if, and only if, after an $\mathbb{F}_q$-projectivity, one of the following conditions holds:
\begin{itemize}
    \item[(i)] $\mathbb{F}_4 \subseteq \mathbb{F}_q$ and $y^{q - 1} = x^{\frac{2(q - 1)}{3}} + x^{\frac{q - 1}{3}} + 1$;
    \item[(ii)] $p = 3$ and $y^{q - 1} = x^{q - 1} + x^{\frac{q - 1}{2}} + 1$.
\end{itemize}
\end{proposition}

\begin{proof}
Since $\#V_g = 1 + \dfrac{q - 1}{a}$, we must have $a = q - 1$, so $g(y) = y^{q - 1}$ and $V_g = \{0, 1\}$. Assume first that $p = 2$. Then, by Propositions~\ref{borges-MVSPcondition} and~\ref{Vf=2}, $\mathcal{C}$ is $\mathbb{F}_q$-Frobenius nonclassical if, and only if,
$$
f(x) = \alpha x^{\frac{2(q - 1)}{3}} + \alpha \beta x^{\frac{q - 1}{3}} + \gamma,
$$
with $\beta^3 = 1$ and $V_f = V_g$. Since $\{0,1\} = V_f = \{\alpha \beta^2 + \gamma, \gamma\}$, we get $\gamma = 1$ and $\alpha = \beta$. Since $\beta^2$ is in the image of the norm function $x \mapsto x^{\frac{q-1}{3}}$, we can apply a suitable $\mathbb{F}_q$-projectivity to put $\mathcal{C}$ in the desired form. An analysis with the derivative of $f(x) = x^{2(q-1)/3} + x^{(q-1)/3} + 1$ shows that this polynomial does not have a multiple root. From Proposition \ref{irreducible}, we conclude that $\mathcal{C}$ is an algebraically irreducible curve.

Now suppose that $p > 2$. If the irreducible components of $\mathcal{C}$ are $\mathbb{F}_q$-Frobenius nonclassical, then
$$
f(x) = \alpha x^{q - 1} + \alpha \beta x^{\frac{q - 1}{2}} + \gamma,
$$
with $\beta^2 = 1$ and $V_f = V_g$. In this case, $\{0,1\} = V_f = \{2\alpha + \gamma, \gamma\}$ implies $\gamma = 1$ and $\alpha = -1/2$. The polynomial $f(x)$ will then be
$$f(x) = -\frac{x^{q - 1}}{2} - \frac{\beta x^{(q - 1)/2}}{2} + 1.$$
Applying Proposition \ref{borges-condition} and comparing the coefficients of $x^{2q - 2}$ on both sides of the associated equation (namely, $-\frac{1}{4}$ on the left-hand side and $\frac{1}{2}$ on the right-hand side) we conclude that the irreducible components of $\mathcal{C}$ are $\mathbb{F}_q$-Frobenius nonclassical if, and only if, $p = 3$. In this case, we have $\mathcal{C}$ as stated, and one can check that 
$$y^{q-1} - x^{q-1} - x^{\frac{q-1}{2}} - 1 = \left(y^{\frac{q-1}{2}}-(x^{\frac{q-1}{2}}-1) \right) \cdot \left(y^{\frac{q-1}{2}}+(x^{\frac{q-1}{2}}-1) \right),$$
so the quadrinomial curve is the product of two $\mathbb{F}_q$-Frobenius nonclassical Fermat curves. Since this curve is reducible, it is not included in the classification of irreducible curves in Theorem B.

%Finally, for $p > 3$, we still need to check the existence of $T \in \mathbb{F}_q[z]$ such that $T(g(y)) = (y^q - y)g'(y)$ and $T(f(x)) = (x^q - x)f'(x)$. By comparing the polynomial in $y$, one can check that $T(z) = -z^2 + z$. Applying $T$ in $f(x)$ now yields $p = 3$, which is not the case.
\end{proof}

\begin{proposition}
Let $\mathcal{C} \colon g(y) = f(x)$ be a quadrinomial curve of type (ii) such that $\# V_f > 2$ and whose irreducible components are defined over $\mathbb{F}_q$. Then these components are $\mathbb{F}_q$-Frobenius nonclassical if, and only if, it is defined by one of the following plane equations:
\begin{itemize}
    \item[(i)] $y^{(p^{\ell}+1)t} = x^{p^{\ell}t} + x^t + 1$, where $t = \dfrac{q - 1}{p^{2\ell} - 1}$;
    
    \item[(ii)] $y^{2t} = (ax^t + b)^2$, where $p > 2,$ $t = \dfrac{q - 1}{p^m - 1}$ and $a,b \in \mathbb{F}_{p^m}.$
\end{itemize}
\end{proposition}

\begin{proof}
If $\mathcal{C}$ is $\mathbb{F}_q$-Frobenius nonclassical, then by Proposition~\ref{borges-MVSPcondition} we have $g(y) = y^a$ with $a \mid (q - 1)$ and 
$$
1 + \frac{q - 1}{a} = \# V_g = \# V_f,
$$ 
where $f(x)$, by Propositions~\ref{Vf>2} and~\ref{Vf>22}, must have of one of the following forms:
\begin{itemize}
    \item $f(x) = \alpha x^{p^{\ell}t} + \alpha \beta x^t + \gamma$, where $t = \dfrac{q - 1}{p^{s\ell} - 1}$, $\beta \in \mathbb{F}_{p^{s\ell}}$, $\gamma \in \mathbb{F}_q$, and $N_{\mathbb{F}_{p^{s\ell}}/\mathbb{F}_{p^{\ell}}}(\beta) = (-1)^s$ with $s \geq 2$;

    \item $f(x) = \alpha x^{2t} + \alpha \beta x^t + \gamma$, where $p > 2$, $t = \dfrac{q - 1}{p^m - 1}$, $\beta \in \mathbb{F}_{p^m}$, and $\gamma \in \mathbb{F}_q^*$.
\end{itemize}

We start with the first case. Here, $\# V_f = p^{(s - 1)\ell}$, so from the equality $\#V_g = \#V_f$, we deduce that $a = \dfrac{q - 1}{p^{(s - 1)\ell} - 1}$. Then, by Proposition~\ref{borges-condition}, $\mathcal{C}$ is $\mathbb{F}_q$-Frobenius nonclassical if, and only if, $s = 2$, $\alpha = (\alpha \beta)^{p^{\ell}}$, $\alpha \beta \in \mathbb{F}_{p^{2\ell}}^*$, and $\gamma \in \mathbb{F}_{p^{\ell}}^*$. Since the norm map is surjective, up to projective equivalence the curve $\mathcal{C}$ takes the form:
$$
y^{(p^{\ell} + 1)t} = x^{p^{\ell}t} + x^t + \gamma,
$$
where $t = \dfrac{q - 1}{p^{2\ell} - 1}$ and $\gamma \in \mathbb{F}_{p^{\ell}}^*$. Since $\gamma$ is in the image of the norm functions $x \mapsto x^t$ and $y \mapsto y^{(p^\ell + 1)t}$, we can rescale both $x$ and $y$ to put $\mathcal{C}$ in the desired form.

Now consider the second case. Here, $\#V_f = 1 + \dfrac{p^m - 1}{2}$, hence $a = 2t$. Since $V_g \subseteq \mathbb{F}_{p^m}$ and we require $V_f = V_g$, it follows that $\alpha, \gamma \in \mathbb{F}_{p^m}$. 
Let $\tilde{f}(x) := \alpha x^2 + \alpha \beta x + \gamma \in \mathbb{F}_{p^m}[x]$, and denote $V_{\tilde{f}}(\mathbb{F}_{p^m})$ and $V_{y^2}(\mathbb{F}_{p^m})$ as the respective value sets of $\tilde{f}$ and $y^2$ over $\mathbb{F}_{p^m}$. Note that both $\tilde{f}$ and $y^2$ are MVSPs over $\mathbb{F}_{p^m}$, and since $V_f = V_g$, we also have $V_{\tilde{f}}(\mathbb{F}_{p^m}) = V_{y^2}(\mathbb{F}_{p^m})$.

If $p^m = 3$, we have $\# V_f = 2$, so this case is not considered here. On the other hand, if $p^m > 3$, by \cite[Lemma 2.6]{borges-separated} we have $\tilde{f}(x) = (ax + b)^2$ for some $a, b \in \mathbb{F}_{p^m}^*$, i.e., $f(x) = (ax^t + b)^2$ with $a,b \in \mathbb{F}_{p^m}$. Thus,  $\mathcal{C}$ is defined by  the following plane equation:
$$
y^{2t} = (ax^t + b)^2,
$$
where $p > 2$ and $t = \dfrac{q - 1}{p^m - 1}$ and $a,b \in \mathbb{F}_{p^m}^*$.
\end{proof}

\begin{remark}
    An analysis with the derivative of $f(x) = x^{p^{\ell}t} + x^t + 1$ shows that this polynomial does not have a multiple root. From Proposition \ref{irreducible}, the first curve of the previous proposition is algebraically irreducible.
    
    On the other hand, one can see that the second curve of the previous proposition is reducible. Indeed, it is the product of two Frobenius nonclassical Fermat curves
    $$\mathcal{C}: (y^t - (ax^t + b))(y^t + (ax^t + b)) = 0.$$ Since this curve is reducible, it is not included in the classification of irreducible curves in Theorem B.
\end{remark}

\section{Quadrinomial curves of type (iii)}\label{section:quadrinomial3}

Finally we will study quadrinomial curves of type (iii), that is, curves defined by the plane equation
$$
y^a = \alpha x^b + \beta x^c + \gamma x^d, \quad b > c > d.
$$
The Frobenius nonclassicality of these curves will not be given entirely via MVSPs, and at the end of this section, we will complete the classification of all Frobenius nonclassical quadrinomial curves with separated variables. Throughout this section, we assume $f(x) = \alpha x^{b} + \beta x^c + \gamma x^d$ and $g(y) = y^a$.

Now, as before, if $\# V_f = 1$, there are no $\mathbb{F}_q$-Frobenius nonclassical curve of type (iii). Finally, if $\# V_f = 2$, we obtain $a = q - 1$. Direct computations using Proposition~\ref{borges-condition} imply the following result.

\begin{proposition}
Let $\mathcal{C}$ be a quadrinomial curve of type (iii) such that $\# V_f = 2$ and whose irreducible components are defined over $\mathbb{F}_q$. Then these components are $\mathbb{F}_q$-Frobenius nonclassical if, and only if, $q = 3^{2n}$ and $\mathcal{C}$ is given by
$$\mathcal{C}: y^{q-1} = x^{3(q-1)/4} - x^{2(q-1)/4} + x^{(q-1)/4},$$
or $p = 2$ and $\mathcal{C}$ is defined by one of the following plane equations:
\begin{itemize}
    \item[(i)] $y^{q-1} = x^{4t} + x^{2t} + x^t$, where $q = 2^{3n}$ and $t = \frac{q - 1}{7}$;
    
    \item[(ii)] $y^{q-1} = x^{6t} +  x^{5t} + x^{3t}$, where $q = 2^{3n}$ and $t = \frac{q - 1}{7}$;
    
    \item[(iii)] $y^{q-1} = x^{3t} + x^{2t} + x^t$, where $q = 2^{2n}$ and $t = \frac{q - 1}{3}$.
\end{itemize}
\end{proposition}

\begin{proof}
 Since $0$ is a root of multiplicity $d$ of $f(x)$, by Corollary~\ref{borges-corollary} (vi) we have
$d \equiv -1 \pmod{p}$. By Proposition \ref{borges-condition}, we have the following equation:
\begin{equation}\label{eqm}
(\alpha x^{b} + \beta x^c + \gamma x^d) - (\alpha x^{b} + \beta x^c + \gamma x^d)^2
= (x^q - x)(\alpha b x^{b-1} + \beta c x^{c-1} - \gamma  x^{d-1}).
\end{equation}
We divide our analysis into three cases:
\\

\noindent (a) \textbf{Case $p \mid b$ and $p \mid c$:}
By comparing the term of highest degree on both sides of Equation~\eqref{eqm}, we obtain $2b = q - 1 + d$ (so
$d \equiv 1 \pmod{p}$ and $p = 2$) and $\alpha^2 = \gamma$. Hence, Equation~\eqref{eqm} (after canceling the terms of highest and lowest degree on both sides) becomes
\begin{equation*}
\alpha x^b + \beta x^c + \beta^2 x^{2c} + \gamma^2 x^{2d} = 0.
\end{equation*}

This yields
\[
\begin{cases}
b = 2c \quad \text{and} \quad \alpha = \beta^2, \\
c = 2d \quad \text{and} \quad \beta = \gamma^2.
\end{cases}
\]
Consequently, $\gamma \in \mathbb{F}_{8}^{*}$ and $8d = 2b = q - 1 + d$.
Therefore, after a suitable change of variables, $\mathcal{C}$ has equation~$(i)$ in the statement.
\\

\noindent (b) \textbf{Case $p \mid b$ and $p \nmid c$:} By comparing the term of highest degree on both sides of Equation~\eqref{eqm}, we now obtain $\beta = - \alpha^2$ and $2b = q - 1 + c$. In particular, $c \equiv 1 \pmod{p}$. By comparing the terms with the second lowest degrees, we see that $2d \geq c$. 

If $c < 2d$, then we must have $p = 2$. Equation \eqref{eqm} now becomes
$$\alpha x^b + \alpha^4 x^{2c} + \gamma^2 x^{2d} = \gamma x^{q - 1 + d}.$$
Therefore, $b = 2d$, $2c = q - 1 + d$, $\alpha = \gamma^2$ and $\gamma = \alpha^4$. After a suitable change of variables, $\mathcal{C}$ has equation~$(ii)$ in the statement.

If $c = 2d$, we must have $p = 3$ (since $d \equiv -1 \pmod{p}$) and Equation \eqref{eqm} now becomes
$$\alpha x^b - (\alpha^2 + \gamma^2) x^{2d} - \alpha^3 x^{b + 2d} + \alpha \gamma x^{b + d} - \alpha^4 x^{4d} - \alpha^2 \gamma x^{3d} = \alpha^2 x^{2d} - \gamma x^{q - 1 + d},$$
and this is only possible if
$$b = 3d,\quad \alpha = \alpha^2\gamma,\quad \alpha^2 = \gamma^2,\quad \alpha^3 = \gamma,\quad d = \frac{q-1}{4},\quad \alpha\gamma = \alpha^4.$$
This implies that $\alpha^4 = 1$. In particular, $\alpha$ is in the image of the function $x \mapsto x^{(q-1)/4}$, so we can rescale $x$ to put $\mathcal{C}$ in the desired form.

Note that, in order to have $4 \mid (q - 1)$, we must have $q = 3^{2n}$. In particular, we also have $8 \mid (q - 1)$. Therefore, decompose the curve $\mathcal{C}$ as
$$\mathcal{C}: (y^{(q - 1)/2} + (x^{3(q - 1)/8} + x^{(q - 1)/8}))(y^{(q - 1)/2} - (x^{3(q - 1)/8} + x^{(q - 1)/8})) = 0.$$
Since $\mathcal{C}$ is reducible,  it is not included in the classification of irreducible curves in Theorem B.
\\

\noindent (c) \textbf{Case $p \nmid b$:} By Corollary~\ref{borges-corollary} (ii), we have $b = q - 1$.
Comparing the term of highest degree on both sides of Equation~\eqref{eqm}, we obtain $\alpha = 1$.
Thus, this equation becomes
\begin{equation*}
\beta x^c - 2\beta x^{q-1+c} - 2\gamma x^{q-1+d}
- \beta^2 x^{2c} - 2\beta\gamma x^{c+d} - \gamma^2 x^{2d}
= \beta c x^{q-1+c} - \gamma x^{q-1+d} - \beta c x^c.
\end{equation*}

There is only the term $-2\beta\gamma x^{c+d}$ of degree $c+d$ on both sides of the previous equation.
Thus, $p = 2$.
Now, this equation becomes
\begin{equation*}
\beta x^c + \beta^2 x^{2c} + \gamma^2 x^{2d}
= \beta c x^{q-1+c} + \gamma x^{q-1+d} + \beta c x^c.
\end{equation*}

This yields $c \equiv 0 \pmod{p}$.
Otherwise, we would have $2d = q - 1 + d$, which is not possible.
Hence,
\[
\begin{cases}
2c = q - 1 + d \quad \text{and} \quad \beta^2 = \gamma, \\
c = 2d \quad \text{and} \quad \beta = \gamma^2.
\end{cases}
\]
Thus, $\gamma \in \mathbb{F}_4^{*}$, and after a change of variables, $\mathcal{C}$ has equation~$(iii)$ in the statement.
\end{proof}

\begin{remark}
    One can check that the curves of items (i) - (iii) from the previous proposition are irreducible. An analysis with the derivatives of
    $$x^{4t} + x^{2t} + x^t,$$
    $$x^{6t} + x^{5t} + x^{3t} = x^{3t}(x^{3t} + x^{2t} + 1)$$
    and
    $$x^{3t} + x^{2t} + x^t = x^t(x^{2t} + x^t + 1)$$
    shows that the only multiple root of these polynomials is $0$. An analysis with the degree of these polynomials now show that they have simple roots in $\overline{\mathbb{F}_q}$. The irreducibility of the curves now comes from Proposition \ref{irreducible}.
\end{remark}

We can now deal with the case $\# V_f > 2$.

\begin{proposition}
Let $\mathcal{C}$ be a quadrinomial curve of type (iii) such that $p \mid b$ and $p \mid c$. Then $\mathcal{C}$ is $\mathbb{F}_q$-Frobenius nonclassical if, and only if,
$$
\mathcal{C}: y^{(p^{2\ell} + p^\ell + 1)t} = x^{p^{2\ell}t} + x^{p^\ell t} + x^t,
$$
where $\ell > 0$ and $t = \dfrac{q - 1}{p^{3\ell} - 1}$.
\end{proposition}

\begin{proof}
Suppose that $\mathcal{C}$ is $\mathbb{F}_q$-Frobenius nonclassical. Then $f(x)$ is an MVSP. Analogously to what we did in Lemma \ref{lem1}, we can differentiate both sides of the identity in Theorem \ref{MVSP-condition} (b), using the fact that $b \equiv c \equiv 0 \pmod{p}$, analyze all possibilities and conclude that $\gamma_0 = 0$ and $\nu = 1$.

We can then proceed as in the proof of Proposition \ref{Vf>2}, but now we will not consider a monic polynomial $f$, because of some relations that (we are yet to find out) the coefficients share. Equation \eqref{eq:1} now reads
$$\sum_{i=0}^m\omega_i(\alpha x^b + \beta x^c + \gamma x^d)^{p^i} = -\omega_0 \gamma(x^{q-1+d} - x^d).$$
An analogous analysis with the terms with highest degrees and how intermediate terms cancel out allow us to conclude that
$$
f(x) = \gamma^{p^{2\ell}} x^{p^{2\ell}t} + \gamma^{p^\ell} x^{p^\ell t} + \gamma x^t,
$$
where $t= \dfrac{q-1}{p^{s\ell} - 1}$ and $\gamma \in \mathbb{F}_{p^{s \ell}}^*$.

Hence, $\# V_f = p^{(s-2)\ell}$. By Proposition \ref{borges-MVSPcondition}, we obtain $g(y) = y^{(q-1)/(p^{(s-2)\ell} - 1)}$. Finally, by Proposition \ref{borges-condition}, the curve $\mathcal{C}$ is $\mathbb{F}_q$-Frobenius nonclassical if, and only if, $s = 3$, which yields $g(y) = y^{(p^{2\ell} + p^\ell + 1)t}$. Since the norm map is surjective, it follows that, after a suitable rescaling of the variable $x$, that $\mathcal{C}$ has the stated form. The irreducibility of this curve comes from Proposition \ref{irreducible}, by noting that the only multiple root of $f(x)$ is $0$.
\end{proof}

Now, let us deal with curves of type (iii) when $p$ does not divide both $b$ and $c$ simultaneously. We will use the following results.

\begin{lemma}\label{lemm1}
Let $\mathcal{C}$ be a curve of type (iii) that is $\mathbb{F}_q$-Frobenius nonclassical. Then $a \equiv d \equiv 1 \pmod{p}$. Moreover, if $p \nmid b$, then $a = b$. Furthermore, if $p \nmid c$, then $c \equiv 1 \pmod{p}$.
\end{lemma}

\begin{proof}
If $p \nmid b$, then $a = b$ by Corollary~\ref{borges-corollary} (ii). Moreover, since $\mathcal{C}$ is $\mathbb{F}_q$-Frobenius nonclassical, $f(x)$ is an MVSP. Suppose first that $\gamma_0 = 0$. Then, by comparing the terms with the lowest degree in both sides of the equation given in Theorem~\ref{MVSP-condition} (f) (that is, $\omega_0 x^d = d\omega_0 x^d$) we obtain $d \equiv 1 \pmod{p}$. By Corollary~\ref{borges-corollary} (vi) (for the $\mathbb{F}_q$ root $0$), it follows that $a \equiv d \equiv 1$ (mod $p$).

Now suppose that $\gamma_0 \neq 0$. Then, by Theorem \ref{MVSP-condition}, $f(x)$ has a simple root in $\mathbb{F}_q$. A combination of itens (iv) and (vi) of Corollary \ref{borges-corollary} imply that $a \equiv d \equiv 1 \pmod{p}$. Finally, if $p \nmid c$, Corollary \ref{borges-corollary} (v) imply that $c \equiv 1 \pmod{p}$.
\end{proof}

The proof of the following lemma is straightforward, and only involves an analysis of the trinomial expansion.

\begin{lemma}\label{lemm2}
Let $F(x) = a_1 x^{b_1} + a_2 x^{b_2} + a_3 x^{b_3} \in \mathbb{F}_q[x]$ be a trinomial such that $F(0) = 0$. Suppose that $F(x)^s$, with $s > 1$, is also a trinomial. Then $s$ must be a power of $p$.
\end{lemma}

%\begin{proof}
%Without loss of generality, suppose $b_1 > b_2 > b_3$ and $d_1 > d_2 > d_3$. Write $s = p^{\ell} t$ with $\ell \geq 0$ and $p \nmid t$. Then
%$$
%c_1 x^{d_1} + c_2 x^{d_2} + c_3 x^{d_3} = F(x)^s = (a_1^{p^{\ell}} x^{b_1 p^{\ell}} + a_2^{p^{\ell}} x^{b_2 p^{\ell}} + a_3^{p^{\ell}} x^{b_3 p^{\ell}})^t.
%$$
%If $t > 1$, then the expansion of the left-hand side includes more than three terms. For instance, the monomials
%$$
%a_1^{p^{\ell}t} x^{b_1 p^{\ell} t}, \quad t a_1^{p^{\ell}(t-1)}a_2^{p^{\ell}} x^{b_1 p^{\ell}(t - 1) + b_2 p^{\ell}}, \quad t a_2^{p^{\ell}}a_3^{p^{\ell}(t-1)} x^{b_2 p^{\ell} + b_3 p^{\ell}(t - 1)}, \quad a_3^{p^{\ell}t} x^{b_3 p^{\ell} t}.
%$$
%This contradicts the assumption that $F(x)^s$ is a trinomial. Therefore, we must have $t = 1$, and hence $s = p^\ell$.
%\end{proof}

\begin{proposition}
Let $\mathcal{C}$ be a quadrinomial curve of type (iii) such that $p$ does not divide both $b$ and $c$ simultaneously. Then $\mathcal{C}$ is $\mathbb{F}_q$-Frobenius nonclassical if, and only if, it is defined by one of the following plane equations:
\begin{itemize}
    \item[(i)] $y^{(p^{2 \ell} + p^\ell + 1)t} = x^{(p^{2 \ell} + p^\ell)t} + x^{(p^{2 \ell} + 1)t} + x^{(p^\ell + 1)t}$, where $t = \dfrac{q - 1}{p^{3\ell} - 1}$;
    
    \item[(ii)] $y^{(p^\ell+1)t} = \alpha x^{(p^\ell+1)t} + x^{p^{\ell}t} + x^{t}$, where $t = \dfrac{q - 1}{p^{2\ell} - 1}$, and $\alpha \in \mathbb{F}_{p^{\ell}}^*$.
\end{itemize}
\end{proposition}

\begin{proof}
Let us divide the proof into three cases. 
\\

\noindent (a) \textbf{Case $p \mid b$ and $p \nmid c$:} By Proposition~\ref{borges-condition} and Lemma~\ref{lemm1}, we have the following equation:
\begin{equation*}
(\alpha x^b + \beta x^c + \gamma x^d)^{\frac{q-1}{a} + 1}
= \beta x^{q-1+c} + \gamma x^{q-1+d} + \alpha x^b.
\end{equation*}

By Lemma~\ref{lemm2} and comparing the terms on both sides of the last equation, we obtain
\[
\frac{q - 1}{a} + 1 = p^\ell, \quad
b p^\ell = q + c - 1, \quad
c p^\ell = q + d - 1, \quad
d p^\ell = b.
\]

and
$$
\alpha^{p^\ell} = \beta, \quad \beta^{p^\ell} = \gamma, \quad \gamma^{p^\ell} = \alpha.
$$
In particular, $\gamma \in \mathbb{F}_{p^{3\ell}}^*$, and we can rescale $x$ to conclude that $\mathcal{C}$ is defined by the plane equation of item (i).
\\

\noindent (b) \textbf{Case $p \nmid b$ and $p \mid c$:}
Using arguments analogous to those of the previous case, and noting that $a = b$ by Lemma~\ref{lemm1}, we conclude that the irreducible components of $\mathcal{C}$ are $\mathbb{F}_q$-Frobenius nonclassical if, and only if, it is defined by an equation of the form given in item~$(ii)$.
\\

\noindent (c) \textbf{Case $p \nmid b$ and $p \nmid c$:}
Using the same arguments as in the previous cases leads to $b = c = d$, which is a contradiction. Therefore, in this case, there are no $\mathbb{F}_q$-Frobenius nonclassical curves.

The irreducibility of both curves comes from Proposition \ref{irreducible}, after noting that the only multiple root of 
$$x^{(p^{2 \ell} + p^\ell)t} + x^{(p^{2 \ell} + 1)t} + x^{(p^\ell + 1)t} = x^{(p^{\ell}+1)t}(x^{(p^{2\ell}-1)t} + x^{(p^{2\ell}-p)t} + 1)$$
and
$$\alpha x^{(p^\ell+1)t} + x^{p^{\ell}t} + x^{t} = x^t(\alpha x^{p^{\ell}t} + x^{(p^{\ell}-1)t} + 1)$$
is $0$.
\end{proof}

\section{Acknowledgments}
\noindent Tiago Aprigio was supported by FAPESP (Brazil), grant 2023/07508-1. João Paulo Guardieiro was supported by FAPESP (Brazil), grant 2024/19443-4. We would also like to thank the anonymous reviewers for their suggestions on how to improve our paper.


\begin{thebibliography}{99}
\bibitem{beelen-pelikaan}Beelen, Peter, and Ruud Pellikaan. \textit{The Newton polygon of plane curves with many rational points}. Designs, Codes and Cryptography 21.1/3 (2000): 41-67.

\bibitem{borges}
Borges, H. On multi-frobenius non-classical plane curves. 
\textit{Archiv der Mathematik} 93, 6 (2009), 541--553.

\bibitem{borges-separated}
Borges, H. Frobenius nonclassical components of curves with separated variables. 
\textit{Journal of Number Theory} 159 (2016), 402--425.

% \bibitem{borges-conceicao}
% Borges, H., and Conceição, R. On the characterization of minimal value set polynomials. 
% \textit{Journal of Number Theory} 133, 6 (2013), 2021--2035.

\bibitem{borges-homma}
Borges, H., and Homma, M. Points on singular frobenius nonclassical curves. 
\textit{Bulletin of the Brazilian Mathematical Society, New Series} 48, 1 (2017), 93--101.

\bibitem{borges-reis1}
Borges, H., and Reis, L. Minimal value set polynomials over fields of size $p^3$. 
\textit{Proceedings of the American Mathematical Society} 149, 9 (2021), 3639--3649.

\bibitem{borges-reis}
Borges, H., and Reis, L. Minimal value set polynomials, 2025. 
\texttt{https://arxiv.org/abs/2508.07113}.

\bibitem{carlitz61}
Carlitz, L., Lewis, D. J., Mills, W. H., and Straus, E. G. Polynomials over finite fields with minimal value sets. 
\textit{Mathematika} 8 (1961), 121--130.

\bibitem{chow-et-al}
Chou, Wun-Seng; Gomez-Calderon, Javier; Mullen, Gary L. Value sets of Dickson polynomials over finite fields. \textit{Journal of Number Theory}, v. 30, n. 3, p. 334-344, 1988.

\bibitem{garcia1}
Garcia, A., and Voloch, J. Fermat curves over finite fields. 
\textit{Journal of Number Theory} 30, 3 (1988), 345--356.

% \bibitem{gomez}
% Gomez-Calderon, J. A note on polynomials with minimal value sets over finite fields. 
% \textit{Mathematika} 35, 1 (1988), 144--148.

% \bibitem{gomez1}
% Gomez-Calderon, J., and Madden, D. J. Polynomials with small value sets over finite fields. 
% \textit{Journal of Number Theory} 28, 2 (1988), 167--188.

\bibitem{guardieiro}
Guardieiro, J. P. Counting rational points with Stöhr-Voloch theory and Tate-Shafarevich results. 
PhD thesis, Universidade de São Paulo, São Carlos, Brasil, 2025. 
\texttt{https://doi.org/10.11606/T.55.2025.tde-05082025-170919}.

\bibitem{hefez}
Hefez, A., and Voloch, J. F. Frobenius non classical curves. 
\textit{Archiv der Mathematik} 54, 3 (1990), 263--273.

\bibitem{Mills1964}
Mills, W. H. Polynomials with minimal value sets. 
\textit{Pacific Journal of Mathematics} 14 (1964), 225--241.

\bibitem{handbook}
Mullen, G. L., and Panario, D. \textit{Handbook of finite fields}, vol. 17. 
CRC press Boca Raton, 2013.

\bibitem{nie}
Nie, M. Zeta functions of trinomial curves and maximal curves. 
\textit{Finite Fields and Their Applications} 39 (2016), 52--82.

\bibitem{schmidt}
Schmidt, W. M. \textit{Equations over finite fields: an elementary approach}, vol. 536. 
Springer, 2006.

\bibitem{StohrVol}
Stöhr, K. O., and Voloch, J. F. Weierstrass points and curves over finite fields. 
\textit{Proceedings of the London Mathematical Society} 52, 3 (1986), 1--19.

% \bibitem{wan}
% Wan, D. Q., Shiue, P. J.-S., and Chen, C. S. Value sets of polynomials over finite fields. 
% \textit{Proceedings of the American Mathematical Society} 119, 3 (1993), 711--717.

\end{thebibliography}
\end{document}